\pdfoutput=1
\documentclass[11pt,a4paper]{article}

\usepackage[margin=1in]{geometry}
\usepackage{amsmath,amssymb,amsfonts,mathtools,amsthm,mathrsfs,bbm,booktabs,enumitem,tikz,tkz-euclide,graphicx,xcolor,ifpdf,nicefrac}
\usetikzlibrary{calc}
\usepackage[numbers,sort&compress]{natbib}
\usepackage[hidelinks]{hyperref}

\ifpdf
  \DeclareGraphicsExtensions{.pdf,.png,.jpg}
\else
  \DeclareGraphicsExtensions{.eps}
\fi
\graphicspath{{./}}

\definecolor{mygray}{gray}{0.90}
\definecolor{mygreen}{rgb}{0,0.5,0}

\hypersetup{
  pdftitle={Finite element and box-method discretizations for fractional elliptic problems with quadrature and mass lumping},
  pdfauthor={Kelvin J. R. Almeida-Sousa, David Bolin, and Alexandre B. Simas}
}

\providecommand{\email}[1]{\href{mailto:#1}{\texttt{#1}}}
\providecommand{\funding}[1]{#1}
\title{Finite element and box-method discretizations for fractional elliptic problems with quadrature and mass lumping\thanks{\funding{This work was supported by King Abdullah University of Science and Technology (KAUST) under Award No.~ORFS-CRG11-2022-5015.}}}
\author{Kelvin J. R. Almeida-Sousa\thanks{Statistics Program, King Abdullah University of Science and Technology, Saudi Arabia (\email{kelvinjhonson.silva@kaust.edu.sa}, \email{david.bolin@kaust.edu.sa}, \email{alexandre.simas@kaust.edu.sa}).}
\and David Bolin\footnotemark[2]
\and Alexandre B. Simas\footnotemark[2]}
\date{}
\newenvironment{keywords}{\par\smallskip\noindent\textbf{Keywords: }}{\par}
\newenvironment{MSCcodes}{\par\smallskip\noindent\textbf{MSC (2020): }}{\par\medskip}

\theoremstyle{plain}
\newtheorem{theorem}{Theorem}[section]
\newtheorem{lemma}[theorem]{Lemma}
\newtheorem{proposition}[theorem]{Proposition}
\newtheorem{corollary}[theorem]{Corollary}
\theoremstyle{definition}
\newtheorem{definition}[theorem]{Definition}
\newtheorem{assumption}[theorem]{Assumption}
\newtheorem{example}[theorem]{Example}
\theoremstyle{remark}
\newtheorem{remark}[theorem]{Remark}

\numberwithin{equation}{section}

\DeclareMathOperator{\divop}{div}
\DeclareMathOperator{\sign}{sign}

\newcommand{\cA}{\mathcal A}
\newcommand{\cB}{\mathcal B}

\newcommand{\cH}{\mathcal H}

\newcommand{\cS}{\mathcal S}
\newcommand{\cT}{\mathcal T}

\newcommand{\cV}{\mathcal V}

\newcommand{\cZ}{\mathcal Z}
\newcommand{\R}{\mathbb R}
\newcommand{\norm}[1]{\left\lVert #1 \right\rVert}
\newcommand{\seminorm}[1]{\left\lvert #1 \right\rvert}
\newcommand{\ip}[2]{\left(#1,#2\right)}
\newcommand{\innerh}[2]{\left\langle #1,#2\right\rangle_{h}}

\allowdisplaybreaks

\begin{document}

\maketitle

\begin{abstract}
We analyze numerical approximation of the fractional elliptic problem $L^{\beta}u=f$, ${\beta>0}$, where $L$ is a second-order selfadjoint elliptic operator with homogeneous Dirichlet or Neumann boundary conditions. The paper develops a unified conforming piecewise linear framework that covers both the standard finite element discretization and the box-method discretization of fractional powers. The key point is that the discrete fractional operator is defined with respect to an admissible inner product on the trial space. This includes, in particular, the standard $L^{2}$ inner product and the quadrature-based mass-lumped inner product, and we also identify a broader family of admissible inner products interpolating between these two realizations. Within this framework, we show that the mass-lumped choice yields the intrinsic fractional box discretization, namely the one obtained by taking fractional powers of the nonfractional box solution operator. For both the finite element and box-method realizations, we establish error estimates under natural consistency assumptions, making explicit the effect of load quadrature in the box case. The analysis applies directly to practical schemes and is supported by numerical experiments in one and two space dimensions.
\end{abstract}

\begin{keywords}
fractional elliptic operators, box method, finite volume method, mass lumping, quadrature error, finite element method
\end{keywords}

\begin{MSCcodes}
35R11, 65N08, 65N15, 65N30, 47A60
\end{MSCcodes}

\section{Introduction}

We study numerical approximation of the fractional elliptic problem
\begin{equation}\label{eq:model}
\begin{cases}
L^{\beta}u=f & \text{in } \cS,\\
\cB u =0 & \text{on } \partial\cS,
\end{cases}
\end{equation}
where $\cS\subset\R^{d}$ is a bounded convex interval, polygon, or polyhedron, $d\in\{1,2,3\}$, $\beta>0$, $\cB$ denotes either homogeneous Dirichlet or homogeneous Neumann boundary conditions, and the operator is  
${L=-\divop(A\nabla)+\kappa^{2}}$.
Here $A\in W^{1,\infty}(\cS;\R^{d\times d})$ is symmetric and uniformly elliptic, and $\kappa\in W^{1,\infty}(\cS)$ satisfies $\kappa\ge \kappa_{0}>0$. The fractional powers of $L$ are defined by spectral functional calculus, and under these assumptions \eqref{eq:model} has a unique solution $u_{\beta}=L^{-\beta}f$.

Problems of this form are now standard in numerical analysis, stochastic modeling, and scientific computing. Fractional powers of elliptic operators appear in nonlocal diffusion \cite{du2020numerical,bonito2015numerical}, in regularity theory for parabolic equations \cite{stinga2017regularity}, in rational and sinc approximations of semigroups and resolvents \cite{bonito2015numerical,harizanov2020numerical}, and in Gaussian field models such as the Whittle--Mat\'ern class and its generalizations; see, for example, \cite{lindgren2011explicit,bolin2020rational,bolin2024covariance}. In all of these settings, the central numerical question is how to approximate $L^{-\beta}$ accurately and computationally efficiently.

At the continuous level, the problem is based on the coercive bilinear form
\begin{equation}\label{eq:continuous-bilinear}
a(u,v)=\int_{\cS} A\nabla u\cdot\nabla v\,dx+\int_{\cS}\kappa^{2}uv\,dx,
\qquad u,v\in\cV,
\end{equation}
where $\cV=H_{0}^{1}(\cS)$ in the Dirichlet case and $\cV=H^{1}(\cS)$ in the Neumann case. On a conforming simplicial mesh $\cT_h$, letting $V_h\subset\cV$ denote the standard continuous piecewise affine space, the classical $P_1$ finite element method (FEM) for the nonfractional problem seeks $u_h^{\mathrm{fe}}\in V_h$ such that
\begin{equation}\label{eq:intro-fem}
a(u_h^{\mathrm{fe}},\chi)=(f,\chi), \qquad \chi\in V_h.
\end{equation}
This is also the natural starting point for fractional discretization: once a discrete elliptic operator has been defined on $V_h$, its fractional powers may be defined by finite-dimensional spectral calculus.

Indeed, several numerical approaches have been developed for fractional elliptic problems, which typically combine FEM approximation with an approximation of the fractional power, including quadrature-based methods \cite{bonito2015numerical}, rational approximation methods \cite{bolin2020rational,bolin2024covariance}, and higher-dimensional extension techniques \cite{glusa2021error} based on the framework introduced by \cite{caffarelli2007extension}.

The issue we are investigating in this work
 is that the idealized formulation \eqref{eq:intro-fem} is usually not the one that is actually assembled. In computation, the stiffness matrix, the mass matrix, and the load vector are typically obtained by numerical quadrature. Thus the implemented scheme often replaces the exact bilinear form by a quadrature-based bilinear form and replaces the exact $L^{2}$ pairing by a quadrature-based inner product; in particular, the latter includes the classical mass-lumping construction, in which the mass matrix is replaced by a diagonal quadrature-based approximation, see e.g., \cite{brenner, ciarlet2002finite, Thomée2006}. For the nonfractional problem, such modifications are classical variational crimes in the sense of \cite[Chapter~4]{strangfix2008}.
For fractional powers, however, they play a more structural role. Once the discrete approximation is defined by taking a fractional power of a finite-dimensional operator, the bilinear form used to define that operator and the inner product used to realize it on $V_h$ both affect the operator whose fractional power is being computed. It is therefore unclear how the theory developed in the nonfractional case transfers to the fractional problem.

This motivates the present work, where we analyze the approximation that is actually computed, rather than the ideal exact-integration Galerkin discretization. We study how the fractional discrete approximation depends on three ingredients:
\begin{enumerate}[label=(\alph*)]
\item\label{itema} the discrete bilinear form used to approximate the elliptic operator,
\item\label{itemb} the discretization of the load term, and
\item\label{itemc} the inner product used to realize the discrete operator on the trial space.
\end{enumerate}

This issue is especially relevant in statistics. In the SPDE approach \cite{lindgren2011explicit}, Gaussian random fields are represented as solutions to elliptic or fractional elliptic equations driven by Gaussian white noise. In that setting, mass lumping and related approximations of the $L^2$ inner product are important as they replace dense mass matrices by diagonal surrogates, thereby inducing sparse precision operators and enabling scalable inference. The same principle also underlies more recent rational and covariance-based approximations of fractional SPDE models; see \cite{bolin2020rational,bolin2024covariance}.
Thus, the present paper is motivated both by how FEM approximations are assembled in practice and by the need for sparse, implementable fractional discretizations in stochastic modeling.

To treat these issues, it is useful to compare FEM with the finite volume method (FVM) and with the box method. A finite volume method starts from local conservation: one introduces control volumes $C$ and imposes balance laws of the form
\begin{equation}\label{eq:intro-fvm}
-\int_{\partial C}A\nabla u_h\cdot n_C\,dS+\int_C \kappa^2u_h\,dx=\int_C f\,dx
\end{equation}
for every control volume $C$, where $n_C$ is the outward unit normal on $\partial C$; see, e.g., \cite{chatzipantelidis2002,jianguo1998finite}.
This viewpoint naturally incorporates cellwise averaging of the load and gives a direct local interpretation for the nonfractional problem. For fractional powers, however, the situation is more delicate.
One can certainly take the matrix produced by an FVM and apply a fractional power to it.
But this operation is no longer tied directly to the local balance law \eqref{eq:intro-fvm}.
More importantly, the resulting matrix function need not be a consistent discretization of the fractional power of the continuous operator.
Indeed, $L^{-\beta}$ is itself a global spectral object, and its discretization depends not only on the entries of a stiffness matrix, but also on the Hilbert-space realization with respect to which the discrete operator is selfadjoint.
Thus, to obtain a meaningful fractional discretization, one needs more than local conservation: one needs a discrete operator that approximates the continuous selfadjoint realization of $L$ in a compatible inner product, so that the spectral functional calculus is stable under discretization. This motivates the use of conservative discretizations that also admit a global symmetric representation on a conforming space. That keeps the FVM load averaging and local conservation structure of the nonfractional problem, while providing the Hilbert-space and spectral structure needed to define and analyze fractional powers.
This perspective leads to the box method.

The box method is the vertex-centered control-volume discretization with conforming piecewise linear trial functions on the primal mesh and piecewise constant test functions on the associated barycentric dual mesh; see, for example, \cite{bank1987some,hackbusch1989first,hollbacher2021sharp}. Hence, it preserves the conservative structure of a finite volume scheme while still working on the conforming finite element space $V_h$. This makes it especially suitable for the present problem: it allows one to define fractional powers on a conforming discrete space, to incorporate quadrature effects in the elliptic operator, and at the same time to retain control-volume averaging of the load.

More precisely, if $\cZ_h^{\mathrm{act}}$ denotes the active vertex set, $\{\phi_z\}_{z\in\cZ_h^{\mathrm{act}}}$ is the nodal basis of $V_h$, and if $b_z$ is the barycentric dual cell associated with the vertex $z$,  we define
\begin{equation}\label{eq:intro-V0h}
V_{0,h}=\operatorname{span}\{\mathbbm{1}_{b_z}:z\in\cZ_h^{\mathrm{act}}\},
\end{equation}
and the transfer operator
\begin{equation}\label{eq:intro-Q}
Q:V_h\to V_{0,h},
\qquad
Q\chi=\sum_{z\in\cZ_h^{\mathrm{act}}}\chi(z)\mathbbm{1}_{b_z},
\qquad
\chi=\sum_{z\in\cZ_h^{\mathrm{act}}}\chi(z)\phi_z\in V_h.
\end{equation}
Thus $Q\chi$ is the piecewise constant function on the dual mesh that takes the nodal value $\chi(z)$ on $b_z$. Since $\{\phi_z\}_{z\in\cZ_h^{\mathrm{act}}}$ and $\{\mathbbm 1_{b_z}\}_{z\in\cZ_h^{\mathrm{act}}}$ are bases of spaces with the same dimension, $Q$ is a linear isomorphism from $V_h$ onto $V_{0,h}$. Moreover, although the diffusion coefficient $A$ is assumed to belong to $W^{1,\infty}(\cS;\R^{d\times d})$ and is therefore continuous, it is convenient to introduce the box bilinear form for a general bounded matrix field $B\in L^\infty(\cS;\R^{d\times d})$. This allows us to cover, in particular, discontinuous coefficient fields arising from elementwise averaging or quadrature-based approximations. Hence, for a bounded matrix field $B\in L^\infty(\cS;\R^{d\times d})$, we define the box bilinear form
\begin{equation}\label{eq:intro-box-bilinear}
a_h^{B}(\chi,\psi) = -\sum_{z\in\cZ_h^{\mathrm{act}}}\psi(z)\int_{\partial b_z}B\nabla\chi\cdot\eta\,dS
+\int_{\cS}\kappa^{2}Q\chi\,Q\psi\,dx, \qquad \chi,\psi\in V_h,
\end{equation}
where $\eta$ denotes the outward unit normal on $\partial b_z$.
Intuitively, $a_h^{B}$ is obtained by writing the flux-balance equation on each dual control volume, multiplying the balance on $b_z$ by the nodal weight $\psi(z)$, and summing over all active vertices. In this way, the box method is a conservative finite-volume discretization written on the conforming finite element space $V_h$. Its nonfractional version reads: find $u_h^{\mathrm{box}}\in V_h$ such that
\begin{equation}\label{eq:intro-box}
a_h^{A}(u_h^{\mathrm{box}},\chi)=(f,Q\chi), \qquad \chi\in V_h.
\end{equation}

The point of the present paper is that, once fractional powers are introduced, one must keep separate the discrete bilinear form, the load discretization, and the inner product on $V_h$. Our goal is to analyze precisely this dependence.
To do so, we introduce a class of \emph{admissible} inner products on $V_h$, broad enough to include the exact $L^2$ product, the quadrature-based mass-lumped product, and explicit banded families of inner products. We then study fractional finite-element and box-method discretizations in a common operator framework.
Our main contributions are the following.
\begin{enumerate}[label=(\roman*)]
\item We formulate a unified operator framework for FEM and box-method discretizations of $L^{-\beta}$ in which the discrete bilinear form, the load discretization, and the inner product on $V_h$ are kept separate.

\item We introduce a broad class of admissible inner products that includes the exact $L^2$ product, the lumped-mass product, and explicit banded families of inner products, including the interpolation family of \cite{voet2023mathematical}. Thus the theory applies to practical sparse implementations, not only to ideal exact-integration schemes, and in particular to constructions that produce banded mass matrices relevant for efficient computation in SPDE-based statistical models.
\item We show that the lumped-mass inner product is distinguished among all admissible inner products: For this choice, the fractional box discretization can be reconstructed directly from the nonfractional box solution operator.
\item We prove error estimates for both FEM and box-method fractional discretizations under natural second-order consistency assumptions on the bilinear form and the inner product used on $V_h$. The finite-element scheme retains the classical order of convergence for fractional FEM, while the box-method scheme exhibits the expected loss caused by replacing the exact load functional by control-volume averaging. As a consequence, the convergence order is not affected by the efficiency enhancements that may arise from \ref{itema}, \ref{itemb} and \ref{itemc}.
\item We show that the proposed method can be combined with any exponentially convergent approximation method for the discrete operator, such as rational approximation or sinc quadrature, to obtain a fully implementable method with theoretical guarantees.

\end{enumerate}

The analysis combines three ingredients. First, the barycentric box construction yields an exact local flux identity for cellwise constant coefficients, which leads to a coercivity theorem in $d=1,2,3$. Second, the dependence on the auxiliary inner product is handled at the operator level rather than being hidden in matrices. Third, the error analysis separates the perturbation of the bilinear form from that of the load term, making it possible to identify which part of the convergence rate comes from approximating the elliptic operator and which part comes from the box discretization of the right-hand side.

The remainder of the paper is organized as follows. Section~\ref{sec:setting} introduces the mesh notation, the primal and dual meshes, the conforming $P_1$ space, and the nonfractional box method. Section~\ref{sec:discretizations} defines the FEM and box-method fractional discretizations, introduces admissible inner products, and proves the lumped-mass characterization. Section~\ref{sec:error} states the convergence estimates for both discretizations under rough and smooth data and extends them to fully implementable rational approximations. Section~\ref{sec:proofs} collects the more technical proofs and Section~\ref{sec:numerics} illustrates the theory through numerical experiments. 

\section{Finite elements, finite volumes, and the box method}\label{sec:setting}
We begin by introducing the concepts required for the box method in the non-fractional case $\beta = 1$, which is the foundation for the general fractional problem discussed in Section~\ref{sec:discretizations}.

Recall that
$\cV= H^{1}_{0}(\cS)$ for homogeneous Dirichlet conditions and
$\cV= H^{1}(\cS)$ for homogeneous Neumann conditions, and let $\cV'$ denote the topological dual of $\cV$. We write $\cH=L^{2}(\cS)$, $\norm{v}=\norm{v}_{L^{2}(\cS)}$, $\seminorm{v}_{1}=\norm{\nabla v}_{L^{2}(\cS)}$, and $\norm{v}_{H^{1}(\cS)}^{2}=\norm{v}^{2}+\seminorm{v}_{1}^{2}$. Further, we use the notation $\lesssim$ to denote inequalities that hold up to a constant independent of $h$, and $\simeq$ to denote two-sided inequalities.

\subsection{Primal and dual meshes}
Let $\{\cT_{h}\}_{h\downarrow 0}$ be a conforming family of simplicial triangulations of $\overline{\cS}$. For each element $K\in\cT_{h}$, write $h_{K}=\operatorname{diam}K$ and let $\rho_{K}$ be the diameter of the largest ball contained in $K$. The family is \emph{shape regular} if there exists $\sigma_{\mathrm{sr}}>0$ such that
\begin{equation}\label{eq:shape-regular}
\rho_{K}\ge \sigma_{\mathrm{sr}} h_{K}, \qquad K\in\cT_{h},\ h>0.
\end{equation}
We also write $h=\max_{K\in\cT_{h}} h_{K}$. In Section~\ref{sec:error} we will additionally assume quasi-uniformity, namely that there exists $\sigma_{\mathrm{qu}}>0$ such that $\sigma_{\mathrm{qu}} h\le h_{K}\le h$ for all $K\in\cT_{h}$.

Let $\cZ_{h}$ be the set of mesh vertices, let $\cZ_{h}^{\mathrm{bd}}=\cZ_{h}\cap \partial\cS$, and let $\cZ_{h}^{\mathrm{in}}=\cZ_{h}\setminus \cZ_{h}^{\mathrm{bd}}$. The set of active vertices is
\begin{equation*}
\cZ_{h}^{\mathrm{act}}=
\begin{cases}
\cZ_{h}^{\mathrm{in}} & \text{in the Dirichlet case},\\
\cZ_{h} & \text{in the Neumann case}.
\end{cases}
\end{equation*}
The trial space is the conforming piecewise affine space
$V_{h}=\operatorname{span}\{\phi_{z}: z\in\cZ_{h}^{\mathrm{act}}\}\subset \cV$,
where $\phi_{z}$ denotes the standard nodal basis function associated with $z$. Equivalently, $V_{h}$ is the image of the global Lagrange interpolant associated with $\cT_{h}$; see \cite[Definition~3.3.9]{brenner}. Thus $V_{h}$ is the usual $P_{1}$ finite-element space with the boundary condition built into the choice of active vertices.

As explained in the introduction, the box method keeps the conforming trial space $V_h$ and uses barycentric dual cells as control volumes. In order to define those cells, fix $K\in\cT_h$ and let $\cZ(K)$ denote the set of vertices of $K$. For each $z\in \cZ(K)$, we construct a region $\cA_z(K)\subset K$ as follows. We consider all $k$-dimensional faces of $K$, for $0\le k\le d$, that contain $z$, and collect their barycenters. Thus, this collection includes the vertex $z$ itself (viewed as a $0$-dimensional face), the barycenters of all edges containing $z$, the barycenters of all higher-dimensional faces containing $z$, and the barycenter of $K$ itself. If $q_1,\dots,q_{M_h}$ denotes an enumeration of these barycenters other than $z$, we define $\cA_z(K)=\operatorname{co}(z,q_1,\dots,q_{M_h}),$ where $\operatorname{co}(\cdot)$ denotes the convex hull. See Figure \ref{triangleK} for an illustration of this construction. The family $\{\cA_z(K)\}_{z\in\cZ(K)}$ partitions $K$ into $d+1$ subregions of equal $d$-dimensional Lebesgue measure. The dual control volume associated with an active vertex $z$ is then

\begin{equation*}
b_{z}=\bigcup_{K\in\cT_{h}: z\in K} \cA_{z}(K).
\end{equation*}
The set $\mathscr{B}_{h}=\{b_{z}\}_{z\in \cZ_{h}^{\mathrm{act}}}$ forms the vertex-centered dual mesh used by the box method; see also \cite{bank1987some,chatzipantelidis2002,jianguo1998finite}. Figure~\ref{fig:dual-mesh} illustrates the construction for $d=2$.

\begin{figure}[t]
    \centering
    \begin{minipage}{0.45\textwidth}
        \centering
        \begin{tikzpicture}[rotate=-45, scale=1.5]
            \coordinate (A) at (0,0);
            \coordinate (B) at (-0.5,2);
            \coordinate (C) at (1.5,2);

            \coordinate (M_AB) at ($(A)!0.5!(B)$);
            \coordinate (M_BC) at ($(B)!0.5!(C)$);
            \coordinate (M_CA) at ($(C)!0.5!(A)$);

            \tkzInterLL(A,M_BC)(B,M_CA) \tkzGetPoint{B_ABC}
            \tkzInterLL(B_ABC,C)(M_BC,M_CA) \tkzGetPoint{L}

            \draw[dashed] (M_AB) -- (B_ABC);
            \draw[dashed] (M_BC) -- (B_ABC);
            \draw[dashed] (M_CA) -- (B_ABC);

            \draw[thick] (A) -- (B) -- (C) -- cycle;

            \fill[fill=mygray] (B_ABC) -- (M_BC) -- (C) -- (M_CA) -- cycle;

            \fill (B_ABC) circle (1pt) node[below left] {$q_{3}$};
            \fill (C) circle (1pt) node[below right] {$z$};
            \fill (M_BC) circle (1pt) node[above right] {$q_2$};
            \fill (M_CA) circle (1pt) node[below] {$q_1$};
            \fill (B_ABC) node[pin={[pin distance=0.001cm,pin edge={}]-5:$\mathcal{A}_z(K)$}] {};

        \end{tikzpicture}
        \caption{Polygonal region $\mathcal{A}_{z}(K)$, in the case $d=2$, where $q_1$ and $q_2$ are the barycenter (midpoints) of the edges that contain $z$ and $q_{3}$ is the barycenter of $K$ itself.}
        \label{triangleK}
    \end{minipage}
    \hfill
    \begin{minipage}{0.45\textwidth}
        \centering
        \begin{tikzpicture}[rotate=-45]

            \coordinate (A) at (0,0);
            \coordinate (B) at (1.5,0.5);
            \coordinate (C) at (2.5,1.5);
            \coordinate (D) at (3,3);
            \coordinate (E) at (2,4);
            \coordinate (F) at (0.5,3.5);
            \coordinate (G) at (-0.5,2);
            \coordinate (O) at (1.5,2);

            \coordinate (M_OA) at ($(O)!0.5!(A)$);
            \coordinate (M_AB) at ($(A)!0.5!(B)$);
            \coordinate (M_BO) at ($(B)!0.5!(O)$);

            \tkzInterLL(A,M_BO)(B,M_OA) \tkzGetPoint{B_OAB}

            \coordinate (M_OB) at ($(O)!0.5!(B)$);
            \coordinate (M_BC) at ($(B)!0.5!(C)$);
            \coordinate (M_CO) at ($(C)!0.5!(O)$);

            \tkzInterLL(B,M_CO)(C,M_OB) \tkzGetPoint{B_OBC}

            \coordinate (M_OC) at ($(O)!0.5!(C)$);
            \coordinate (M_CD) at ($(C)!0.5!(D)$);
            \coordinate (M_DO) at ($(D)!0.5!(O)$);

            \tkzInterLL(C,M_DO)(D,M_OC) \tkzGetPoint{B_OCD}

            \coordinate (M_OD) at ($(O)!0.5!(D)$);
            \coordinate (M_DE) at ($(D)!0.5!(E)$);
            \coordinate (M_EO) at ($(E)!0.5!(O)$);

            \tkzInterLL(D,M_EO)(E,M_OD) \tkzGetPoint{B_ODE}

            \coordinate (M_OE) at ($(O)!0.5!(E)$);
            \coordinate (M_EF) at ($(E)!0.5!(F)$);
            \coordinate (M_FO) at ($(F)!0.5!(O)$);

            \tkzInterLL(E,M_FO)(F,M_OE) \tkzGetPoint{B_OEF}

            \coordinate (M_OF) at ($(O)!0.5!(F)$);
            \coordinate (M_FG) at ($(F)!0.5!(G)$);
            \coordinate (M_GO) at ($(G)!0.5!(O)$);

            \tkzInterLL(F,M_GO)(G,M_OF) \tkzGetPoint{B_OFG}

            \coordinate (M_OG) at ($(O)!0.5!(G)$);
            \coordinate (M_GA) at ($(G)!0.5!(A)$);
            \coordinate (M_AO) at ($(A)!0.5!(O)$);

            \tkzInterLL(G,M_AO)(A,M_OG) \tkzGetPoint{B_OGA}

            \coordinate (Out_AB) at ($(B_OAB)!1.5!(M_AB)$);
            \coordinate (Out_BC) at ($(B_OBC)!1.5!(M_BC)$);
            \coordinate (Out_CD) at ($(B_OCD)!1.5!(M_CD)$);
            \coordinate (Out_DE) at ($(B_ODE)!1.5!(M_DE)$);
            \coordinate (Out_EF) at ($(B_OEF)!1.5!(M_EF)$);
            \coordinate (Out_FG) at ($(B_OFG)!1.5!(M_FG)$);
            \coordinate (Out_GA) at ($(B_OGA)!1.5!(M_GA)$);

            \draw[thick] (A) -- (B) -- (C) -- (D) -- (E) -- (F) -- (G) -- cycle;

            \fill[fill=mygray] (M_OA) -- (B_OAB) -- (M_OB) -- (B_OBC) -- (M_OC) -- (B_OCD) -- (M_OD) -- (B_ODE) -- (M_OE) -- (B_OEF) -- (M_OF) -- (B_OFG) -- (M_OG) -- (B_OGA) -- cycle;

            \draw[dashed, gray] (M_OA) -- (B_OAB) -- (M_OB) -- (B_OBC) -- (M_OC) -- (B_OCD) -- (M_OD) -- (B_ODE) -- (M_OE) -- (B_OEF) -- (M_OF) -- (B_OFG) -- (M_OG) -- (B_OGA) -- cycle;

            \foreach \point in {A,B,C,D,E,F,G} {
                \draw (O) -- (\point);
            }

            \draw[dashed] (Out_AB) -- (B_OAB);
            \draw[dashed] (Out_BC) -- (B_OBC);
            \draw[dashed] (Out_CD) -- (B_OCD);
            \draw[dashed] (Out_DE) -- (B_ODE);
            \draw[dashed] (Out_EF) -- (B_OEF);
            \draw[dashed] (Out_FG) -- (B_OFG);
            \draw[dashed] (Out_GA) -- (B_OGA);

            \foreach \point in {A,B,C,D,E,F,G,O,M_OA, B_OAB, M_OB, B_OBC, M_OC, B_OCD, M_OD, B_ODE, M_OE, B_OEF, M_OF, B_OFG, M_OG, B_OGA} {
                \fill[black] (\point) circle (1pt);
            }
            \fill (O) circle (1pt) node[below right] {$z$};
        \end{tikzpicture}
        \vspace{-0.1cm}
        \caption{Dual element $b_{z}\in \mathscr{B}_{h}$, corresponding to $z\in \mathcal{Z}_{h}^{in}$ in the case $d=2$.}
        \label{fig:dual-mesh}
    \end{minipage}
\end{figure}

\subsection{The box method, the box bilinear form, and the load operator}
We now specialize the finite volume construction to the barycentric dual mesh. Recall from \eqref{eq:intro-box-bilinear} that the box bilinear form  $a_h^A$ is obtained by transporting local flux balances on the dual cells $b_z$ to the conforming space $V_h$. We now make this  explicit.

For $u_h\in V_h$, the local balance on the dual cell $b_z$ reads
\begin{equation}\label{eq:box-balance}
-\int_{\partial b_z} A\nabla u_h\cdot\eta\,dS+\int_{b_z}\kappa^2Qu_h\,dx=\int_{b_z}f\,dx,
\qquad z\in\cZ_h^{\mathrm{act}}.
\end{equation}
Since $Qu_h$ is constant on $b_z$ with value $u_h(z)$, the reaction term is the standard control-volume contribution. Thus \eqref{eq:box-balance} is the local finite-volume balance underlying the box method. Multiplying the balance on $b_z$ by $\chi(z)$ and summing over all active vertices yields the variational formulation \eqref{eq:intro-box} with bilinear form given by \eqref{eq:intro-box-bilinear}.

We begin with the case in which $B$ is piecewise constant on the primal mesh. In that situation, the flux representation of the box method reduces to a standard domain integral, which makes the symmetry of the form transparent and provides the basic structural identity used below.

\begin{lemma}\label{lem:box-piecewise-constant}
Assume that $B\in L^{\infty}(\cS;\R^{d\times d})$ is constant on each element $K\in\cT_h$. Then, for all $\chi,\psi\in V_h$, the corresponding box bilinear form satisfies
\begin{equation}\label{eq:box-identity}
a_h^{B}(\chi,\psi)
=
\int_{\cS} B\nabla\chi\cdot\nabla\psi\,dx
+
\int_{\cS}\kappa^2Q\chi\,Q\psi\,dx.
\end{equation}
In particular, if $B$ is symmetric, then $a_h^B$ is symmetric.
\end{lemma}

Although stated for piecewise-constant coefficient fields, the identity~\eqref{eq:box-identity} connects directly to the variable-coefficient setting. Indeed, if $A$ is replaced by its elementwise average on each $K\in\cT_h$, one obtains a box bilinear form satisfying an identity of the same type. We record this separately, since it yields an explicit symmetric example within the class of box bilinear forms covered by our fractional framework.

\begin{lemma}\label{lem:avg-coefficient}
Let $\kappa\in W^{1,\infty}(\cS)$ and $A_h$ be the piecewise-constant matrix field defined on each element $K\in\cT_h$ by
$A_h|_K=\frac{1}{|K|}\int_K A(x)\,dx$.
Then, for all $\chi,\psi\in V_h$,
\begin{equation}\label{eq:avg-coefficient}
a_h^{A_h}(\chi,\psi)
=
\int_{\cS} A\nabla\chi\cdot\nabla\psi\,dx
+
\int_{\cS}\kappa^2Q\chi\,Q\psi\,dx.
\end{equation}
In particular, $a_h^{A_h}$ is symmetric.
\end{lemma}

We next state the coercivity result for the physical coefficient field $A$. The proof is given in Section~\ref{sec:proofs}.

\begin{theorem}\label{thm:box-coercive}
Assume $d\in\{1,2,3\}$, $A\in W^{1,\infty}(\cS;\R^{d\times d})$, and that $A(x)$ is symmetric and uniformly elliptic. Let $A_h$ be as in Lemma \ref{lem:avg-coefficient}. Then there exist $c>0$ and $h_0>0$, independent of $h$, such that
\begin{equation}\label{eq:box-coercivity}
a_h^A(\chi,\chi)\ge c\norm{\chi}_{H^1(\cS)}^2 \hbox{ and } a_h^{A_h}(\chi,\chi)\ge c\norm{\chi}_{H^1(\cS)}^2, \quad
\chi\in V_h,
\quad
0<h<h_0.
\end{equation}
\end{theorem}
By the discussion in the introduction regarding the need for symmetric and global bilinear forms in order to naturally define fractional operators, we will consider the symmetric box problem, that is, the box problem associated with $A_h$. By Theorem~\ref{thm:box-coercive}, the symmetric nonfractional box problem
\begin{equation}\label{eq:box-variational}
a_{h}^{A_h}(u_{h},\chi)=\ip{f}{Q\chi}, \qquad \chi\in V_{h},
\end{equation}
has a unique solution for all sufficiently small $h$. We denote the associated nonfractional box solution operator by
\begin{equation}\label{eq:box-solution-operator}
T_{h}^{\mathrm{box}}:L^{2}(\cS)\to V_{h}, \qquad a_{h}^{A_h}(T_{h}^{\mathrm{box}}f,\chi)=\ip{f}{Q\chi},\quad \chi\in V_{h}.
\end{equation}
This is the discrete solution operator of the box method. For the moment we keep this concrete notation, since it is the nonfractional box operator that underlies the later fractional construction. In the next step we pass to a more general operator framework, in which the box method appears as one particular realization.

To express \eqref{eq:box-solution-operator} in operator form, and later extend it to fractional powers, we must specify how the discrete operator is realized on $V_h$. This requires an inner product on $V_h$. Let $\langle\cdot,\cdot\rangle_h$ be an inner product on $V_h$, and write $\norm{\chi}_h=\langle\chi,\chi\rangle_h^{1/2}$. Given a symmetric bilinear form $a_h$ on $V_h$, we define the operators $L_h:V_h\to V_h$, $P_h:L^2(\cS)\to V_h$, and $\Lambda_h:L^2(\cS)\to V_h$ through
\begin{equation}\label{eq:operators-LP}
\langle L_h\chi,\psi\rangle_h=a_h(\chi,\psi),
\qquad
\langle P_hf,\chi\rangle_h=\ip{f}{Q\chi},
\qquad
\langle \Lambda_hf,\chi\rangle_h=\ip{f}{\chi},
\end{equation}
for $\chi,\psi\in V_h$ and $f\in L^2(\cS)$. The operator $P_h$ is the load operator associated with the box method, since it acts against the piecewise constant test function $Q\chi$, whereas $\Lambda_h$ is the corresponding finite-element load operator. In particular, when $a_h=a_h^{A_h}$, the solution operator in \eqref{eq:box-solution-operator} takes the form $T_h^{\mathrm{box}}=L_h^{-1}P_h$. More generally, whenever $a_h$ is coercive on $V_h$, we write $T_h=L_h^{-1}P_h$ for the corresponding nonfractional box solution operator.

In the nonfractional case, the auxiliary inner product enters only through the operator representation. Indeed, the equation $L_hu_h=P_hf$ is equivalent to the variational formulation
$$
a_h(u_h,\chi)=\ip{f}{Q\chi},
\qquad \chi\in V_h,
$$
and therefore the resulting function $u_h\in V_h$ is independent of the particular Riesz map induced by $\langle\cdot,\cdot\rangle_h$. For fractional powers, the situation changes. The operator $L_h^{-\beta}$ is defined through the spectral calculus of $L_h$ as an operator on the Hilbert space $(V_h,\langle\cdot,\cdot\rangle_h)$. Consequently, the inner product is no longer only a representational choice; it becomes part of the discrete fractional operator. We therefore restrict attention to inner products that are uniformly equivalent and consistent with the $L^2$ inner product, and to symmetric bilinear forms that are uniformly consistent with the continuous form $a$. This leads to the following definitions.

\begin{definition}[Admissible inner product]\label{def:admissible-innerproduct}
An inner product $\langle\cdot,\cdot\rangle_h$ on $V_h$ is called \emph{admissible} if the following estimates hold uniformly in $h$:
\begin{enumerate}[label=(\roman*),leftmargin=2.1em]
\item
$\lvert \ip{\chi}{\psi}-\langle \chi,\psi\rangle_h\rvert
\lesssim h^{2}\seminorm{\chi}_{1}\seminorm{\psi}_{1}$ for $ \chi,\psi\in V_h;
$
\item
$
\norm{\chi}_h\simeq \norm{\chi}$ for $
\chi\in V_h.
$
\end{enumerate}
\end{definition}

\begin{definition}[Admissible bilinear form]\label{def:admissible-bilinear}
A symmetric bilinear form $a_h$ on $V_h$ is \emph{admissible} if it is continuous on $V_h\times V_h$ and satisfies the consistency estimate
\begin{equation}\label{eq:bilinear-consistency}
\lvert a(\chi,\psi)-a_h(\chi,\psi)\rvert
\lesssim h^{2}\seminorm{\chi}_{1}\seminorm{\psi}_{1},
\qquad \chi,\psi\in V_h,
\end{equation}
uniformly in $h$.
\end{definition}

\begin{remark}\label{rem:ah-coercive}
Every admissible bilinear form is coercive on $V_h$ for sufficiently small $h$. Indeed, since $a(\cdot,\cdot)$ is coercive on $\mathcal V\supset V_h$, Definition~\ref{def:admissible-bilinear} gives
$$
a_h(\chi,\chi)
\ge a(\chi,\chi)-Ch^2\seminorm{\chi}_1^2
\ge c\norm{\chi}_{H^1(\cS)}^2
$$
for all $\chi\in V_h$ and all sufficiently small $h$.
\end{remark}

The following examples show that the class of admissible inner products includes the standard choices used in practice.

\begin{example}\label{ex:inner-products}
The following three inner products will be used repeatedly.
\begin{enumerate}[label=(\roman*),leftmargin=2.1em]
\item the exact $L^{2}(\cS)$ inner product $\langle \chi,\psi\rangle_{h}=\ip{\chi}{\psi}$,
\item the mixed inner product
$\langle \chi,\psi\rangle_{h}=\ip{\chi}{Q\psi}=\sum_{z,w\in \cZ_{h}^{\mathrm{act}}}\chi(z)\psi(w)\int_{b_{w}}\phi_{z}(x)dx$, where $\phi_{z}$ denotes the nodal basis of $V_{h}$ and $b_{w}$ is the corresponding element of the dual mesh;
\item the lumped-mass inner product
$\langle \chi,\psi\rangle_{h}=\ip{Q\chi}{Q\psi}=\sum_{z\in \cZ_{h}^{\mathrm{act}}}\chi(z)\psi(z)|b_{z}|.$
\end{enumerate}
Their associated norms are uniformly equivalent to the $L^2(\cS)$ norm on $V_h$.
\end{example}

All inner products in Example~\ref{ex:inner-products} are admissible in the sense of Definition~\ref{def:admissible-innerproduct}; see, e.g., \cite{chatzipantelidis2002,ewing2002,Thomée2006}. Further non-trivial examples of admissible inner products will be provided in Proposition~\ref{prop:family-inner-products}.

We next record some elementary properties of the box load operator. These facts are independent of the fractional construction, but they clarify the role of $P_h$ and will be used later. In particular, they show that $P_h$ behaves like a projection only in a very restricted sense.

\begin{proposition}\label{prop:pseudo-proj}
Let $P_{h}$ and $\Lambda_{h}$ be defined by \eqref{eq:operators-LP}. Then
\begin{enumerate}[label=(\roman*),leftmargin=2.1em]
\item $P_{h}=P_{h}\Pi_{0,h}$, where $\Pi_{0,h}$ is the $L^{2}(\cS)$-orthogonal projector onto $V_{0,h}$;
\item If there exists $C>0$, independent of $h$, such that $\norm{Q\chi}\le C\norm{\chi}_{h}$ for all $\chi\in V_{h}$, then $P_{h}:L^{2}(\cS)\to V_{h}$ is bounded and $\norm{P_{h}f}_{h}\le C\norm{f}$;
\item $\ker P_{h}=\{f\in L^{2}(\cS): \Pi_{0,h}f|_{b_{z}}=0 \text{ for all } z\in \cZ_{h}^{\mathrm{act}}\}$ and $\operatorname{ran} P_{h}=V_{h}$;
\item The $L^{2}(\cS)$-adjoint of $P_{h}$ is $P_{h}^{*}=Q\Lambda_{h}$.
\end{enumerate}
\end{proposition}

The proof is routine after writing $\Pi_{0,h}f$ in the dual-cell basis and using the defining relations in \eqref{eq:operators-LP}. The detailed argument is given in Section~\ref{sec:proofs}. Since $\norm{Q\chi}\simeq\norm{\chi}$ on $V_h$, Definition~\ref{def:admissible-innerproduct}(ii) implies that the boundedness assumption in item~(ii) holds for every admissible inner product.

The next corollary makes precise the limited sense in which $P_h$ can be regarded as a projection. 

\begin{corollary}\label{cor:projection-characterization}
The operator $P_{h}$ is a projection if and only if $\langle\chi,\psi\rangle_{h}=\ip{\chi}{Q\psi}$ for all $\chi,\psi\in V_h$. Even in that case, $P_{h}$ is not an orthogonal projection on $L^{2}(\cS)$.
\end{corollary}

Corollary \ref{cor:projection-characterization} follows as a consequence of Proposition~\ref{prop:pseudo-proj}(iv). For a detailed proof, see Section~\ref{sec:proofs}.
The box bilinear form \eqref{eq:intro-box-bilinear} arises from local flux balances on the dual cells, a structure natural for the nonfractional problem.
However, as will be seen in Section~\ref{sec:discretizations}, fractional powers are defined through the associated discrete operator and therefore depend on its global structure rather than on individual control-volume balances.
This is why, in the present setting, we seek box-method bilinear forms that admit a global representation on $V_h$. Moreover, in view of
Definition~\ref{def:admissible-bilinear},
we require these forms to be symmetric, so that the induced discrete operator is selfadjoint and the spectral calculus can be applied to define its fractional powers.
Lemma~\ref{lem:avg-coefficient} provides a first example of this mechanism: when $A$ is replaced by its elementwise average, the local box form \eqref{eq:intro-box-bilinear} can be rewritten in the global form \eqref{eq:avg-coefficient}. More generally, the class of bilinear forms relevant for our fractional box-method includes both such averaged-coefficient constructions and fully quadrature-based global approximations:

\begin{example}\label{ex:global-box-bilinears}
Examples of admissible box-method bilinear forms with global representations are:
\begin{enumerate}[label=(\alph*)]
\item If $A_h$ denotes the elementwise average of $A$, then Lemma~\ref{lem:avg-coefficient} shows that the corresponding box bilinear form admits the global representation \eqref{eq:avg-coefficient}. In particular, it is symmetric and therefore fits naturally into the operator framework used later for fractional powers.

\item A fully quadrature-based example is obtained by using $Q$ as a quadrature rule on the box bilinear form $a_h^{A_h}$. On each element $K$, this rule is exact for affine functions and second-order accurate for smooth integrands. For example, a sufficient condition is $A\in W^{1,\infty}(\cS;\R^{d\times d})$ and $\kappa\in W^{2,\infty}(\cS)$, see e.g.\ \cite{FIX1972525} and \cite[Lemma 2.3]{Liang2001}. This leads to the global bilinear form
\begin{equation}\label{eq:quadrature-example}
a_h(\chi,\psi)
=
Q\bigl(A\nabla\chi\cdot\nabla\psi\bigr)
+
Q\bigl(\kappa^2\chi\psi\bigr).
\end{equation}
This form is written directly as a global quadrature approximation of the continuous bilinear form \eqref{eq:continuous-bilinear}. When symmetric, it also induces a selfadjoint discrete operator, and is therefore equally suitable for the spectral definition of fractional powers. By the standard quadrature-error estimates of \cite{strangfix2008, FIX1972525} and the FEM bounds in \cite{Thomée2006}, this choice satisfies Definition~\ref{def:admissible-bilinear}.
\end{enumerate}
\end{example}

\section{Fractional finite-element and box-method discretizations}\label{sec:discretizations}

By the assumptions on $A$ and $\kappa$, the bilinear form $a$ in \eqref{eq:continuous-bilinear} is continuous and coercive on $\cV$. Hence there exists an isomorphism $L:\cV\to\cV'$ such that
\begin{equation*}
\langle Lu,v\rangle_{\cV'\times\cV}=a(u,v), \qquad u,v\in \cV.
\end{equation*}

When realized on $L^{2}(\cS)$, the operator $L$ is positive, selfadjoint, and has compact inverse. Let $\{(\lambda_{j},\varphi_{j})\}_{j\ge 1}$ be the corresponding $L^{2}(\cS)$-orthonormal eigensystem.

For $s\ge 0$, we define fractional power of $L$ by spectral calculus: for $v\in D(L^{s/2})$,
$L^{s/2}v=\sum_{j\ge1}\lambda_j^{s/2}(v,\varphi_j)\varphi_j$, where
\begin{equation}\label{eq:doth-positive}
\dot H^{s}
:=D(L^{s/2})
=
\left\{v\in L^{2}(\cS): \sum_{j\ge1}\lambda_j^{s}\lvert (v,\varphi_j)\rvert^2<\infty\right\},\quad
\norm{v}_{s}^2
:=
\sum_{j\ge1}\lambda_j^{s}\lvert (v,\varphi_j)\rvert^2.
\end{equation}
For $s<0$, we define $\dot H^{s}$ as the completion of $L^{2}(\cS)$ with respect to the norm in \eqref{eq:doth-positive}; equivalently, $\dot H^{s}$ is the dual of $\dot H^{-s}$ with $L^{2}(\cS)$ as pivot space. This is the scale that will be used throughout the paper. In particular, if $f\in L^{2}(\cS)$ and $\beta>0$, then
$u_{\beta}=L^{-\beta}f\in \dot H^{2\beta}$
and
$\norm{u_{\beta}}_{2\beta}=\norm{f}$.
The error estimates proved later are measured in the norms $\norm{\cdot}_{\delta}$ with $0\le \delta\le 1$.

We now pass from the nonfractional FEM and box schemes to their fractional analogues. The key bookkeeping point is that three ingredients must be kept separate: the symmetric bilinear form $a_{h}$, the auxiliary inner product $\langle\cdot,\cdot\rangle_{h}$ used to realize the discrete operator on $V_{h}$, and the load operator $P_h$. For $\beta=1$ this distinction is largely hidden because the solution is recovered from a variational equation. For $\beta\neq1$, by contrast, the functional calculus is applied directly to the discrete operator, so the choice of Riesz map becomes part of the approximation.

From now on $a_{h}$ denotes an admissible symmetric bilinear form in the sense of Definition~\ref{def:admissible-bilinear}
and $\langle\cdot,\cdot\rangle_{h}$ denotes an admissible inner product in the sense of
Definition~\ref{def:admissible-innerproduct}.
The corresponding discrete operator $L_{h}:V_{h}\to V_{h}$ is defined by \eqref{eq:operators-LP}; by Remark~\ref{rem:ah-coercive}, $L_{h}$ is positive and selfadjoint on the Hilbert space $(V_{h},\langle\cdot,\cdot\rangle_{h})$ for all sufficiently fine meshes, so $L_{h}^{-\beta}$ is well defined by the spectral theorem for every $\beta>0$. Because $L_{h}$ is finite dimensional and positive, the operator $L_{h}^{-\beta}$ also admits the standard holomorphic contour representation, see \cite{Haase2006}.

The finite-element fractional discretization based on an admissible inner product $\langle\cdot,\cdot\rangle_{h}$ is the direct Galerkin analogue of the classical fractional FEM construction and serves as the reference scheme against which the box discretization will be compared. It is given, in terms of our discrete operators, by
\begin{equation}\label{eq:fem-discrete}
\bar u_{\beta,h}=L_{h}^{-\beta}\Lambda_{h}f.
\end{equation}
The box fractional discretization keeps the same discrete operator but replaces the exact $L^{2}(\cS)$ load projector by the box load operator $P_{h}$, thereby retaining the control-volume averaging that characterizes the box method already at the nonfractional level.

\begin{definition}\label{def:box-general}
For $\beta>0$, the \emph{box-method fractional discretization} associated with $a_{h}$ and $\langle\cdot,\cdot\rangle_{h}$ is
\begin{equation}\label{eq:box-general}
u_{\beta,h}=T_{h,\beta}f:=L_{h}^{-\beta}P_{h}f.
\end{equation}
When $\beta=1$, this reduces to the nonfractional discrete solution operator $T_{h}=L_{h}^{-1}P_{h}$.
\end{definition}

\begin{remark}\label{rem:beta1-independent}
For $\beta=1$, the discrete solution operator $T_{h}=L_{h}^{-1}P_{h}$ is independent of the auxiliary inner product on $V_{h}$. Indeed, the equation $L_{h}u_{h}=P_{h}f$ is equivalent to the variational statement
\begin{equation*}
a_{h}(u_{h},\chi)=\ip{f}{Q\chi}, \qquad \chi\in V_{h},
\end{equation*}
which involves neither the Riesz map of $\langle\cdot,\cdot\rangle_{h}$ nor the representation matrices of $L_{h}$ and $P_{h}$. For $\beta\neq1$, the functional calculus is performed on $L_{h}$ itself, and the dependence on $\langle\cdot,\cdot\rangle_{h}$ becomes genuine.
\end{remark}

\subsection{Matrix representation and inner-product dependence}

The preceding remark is transparent in coordinates. Fix the nodal basis $\{\phi_{j}\}_{j=1}^{N_{h}}$ of $V_{h}$ and write $[\chi]\in\R^{N_{h}}$ for the coordinate vector of $\chi\in V_{h}$. Let
$\mathbf M_{h}=(\langle \phi_{j},\phi_{i}\rangle_{h})_{i,j}$ and
$\mathbf K_{h}=(a_{h}(\phi_{j},\phi_{i}))_{i,j}$,
and, for $f\in L^{2}(\cS)$, define
$\mathbf F_{Q}(f)=\bigl(\ip{f}{Q\phi_{i}}\bigr)_{i=1}^{N_{h}}$ and
$\mathbf F(f)=\bigl(\ip{f}{\phi_{i}}\bigr)_{i=1}^{N_{h}}$.
The matrix $\mathbf M_{h}$ depends on the chosen inner product, whereas $\mathbf K_{h}$ depends only on $a_{h}$.

\begin{proposition}\label{prop:matrix-representation}
Let $\beta \in \{1,2,3,\ldots\}$. With the notation above,
\begin{equation}\label{relation}
[L_{h}\chi]=\mathbf M_{h}^{-1}\mathbf K_{h}[\chi],
\qquad
[P_{h}f]=\mathbf M_{h}^{-1}\mathbf F_{Q}(f),
\qquad
[\Lambda_{h}f]=\mathbf M_{h}^{-1}\mathbf F(f).
\end{equation}
Consequently,
\begin{equation}\label{eq:matrix-fractional-representation}
[u_{\beta,h}]=\bigl(\mathbf M_{h}^{-1}\mathbf K_{h}\bigr)^{-\beta}\mathbf M_{h}^{-1}\mathbf F_{Q}(f),
\qquad
[\bar u_{\beta,h}]=\bigl(\mathbf M_{h}^{-1}\mathbf K_{h}\bigr)^{-\beta}\mathbf M_{h}^{-1}\mathbf F(f).
\end{equation}
For $\beta=1$, the dependence on $\mathbf M_{h}$ cancels and
\begin{equation}\label{eq:matrix-nonfractional}
[u_{1,h}]=\mathbf K_{h}^{-1}\mathbf F_{Q}(f).
\end{equation}
\end{proposition}

\begin{proof}
By definition, $\langle L_{h}\chi,\phi_{i}\rangle_{h}=a_{h}(\chi,\phi_{i})$. Writing this relation in terms of coordinates, we obtain
$\mathbf M_{h}[L_{h}\chi]=\mathbf K_{h}[\chi]$,
which gives the first relation in \eqref{relation}. The other two follow from the definitions of $P_{h}$ and $\Lambda_{h}$ in the same way. Formula \eqref{eq:matrix-fractional-representation} is then immediate from Definitions~\ref{def:box-general} and \eqref{eq:fem-discrete}, and \eqref{eq:matrix-nonfractional} follows by setting $\beta=1$.
\end{proof}

Proposition~\ref{prop:matrix-representation} pinpoints the source of the difficulty: the nonfractional discrete solution depends only on $\mathbf K_h$ and the associated load vector, whereas the fractional map depends on the similarity class of $\mathbf M_h^{-1}\mathbf K_h$. In this sense, once fractional powers are introduced, the inner product becomes part of the approximation itself. The same matrix perspective also makes clear that quadrature and mass lumping are not merely cosmetic implementation choices, but give rise to genuinely different discretizations of the same continuous problem.

A useful benchmark is the classical FEM discretization. If $\widehat T_{h}=\widehat L_{h}^{-1}\Pi_{h}$ denotes the nonfractional FEM solution operator associated with the exact form $a$ and the exact $L^{2}(\cS)$ inner product, then $\widehat T_{h}$ is selfadjoint and nonnegative on $L^{2}(\cS)$, and
\begin{equation}\label{eq:classical-fem-identity}
(\widehat T_{h})^{\beta}=\widehat L_{h}^{-\beta}\Pi_{h},
\qquad \beta>0.
\end{equation}
Indeed, $\Pi_{h}$ is the orthogonal projector onto the invariant subspace $V_{h}$ and $\widehat L_{h}$ is positive selfadjoint on $V_{h}$, so the spectral calculus may be performed either on $\widehat T_{h}$ or on $\widehat L_{h}$. Identity \eqref{eq:classical-fem-identity} provides the finite-element analogue of the intrinsic construction we seek for the box method.

The next proposition explains why the operator $QT_{h}$ is the natural box-method counterpart of $\widehat T_{h}$.
\begin{proposition}\label{prop:QT-selfadjoint}
Let $T_{h}=L_{h}^{-1}P_{h}$ be the nonfractional box solution operator. Then $QT_{h}:L^{2}(\cS)\to V_{0,h}\subset L^{2}(\cS)$ is selfadjoint and nonnegative.
\end{proposition}
\begin{proof}
For any $f,g\in L^{2}(\cS)$, the defining relation for $T_{h}$ gives
\begin{equation}\label{rlxc}
\ip{f}{QT_{h}g}=a_{h}(T_{h}f,T_{h}g)=a_{h}(T_{h}g,T_{h}f)=\ip{g}{QT_{h}f},
\end{equation}
which shows that $QT_{h}$ is selfadjoint. Taking $f=g$ in \eqref{rlxc}, we obtain
\begin{equation*}
\ip{f}{QT_{h}f}=a_{h}(T_{h}f,T_{h}f)\ge 0,
\end{equation*}
and the nonnegativity of $QT_{h}$ follows.
\end{proof}

Proposition~\ref{prop:QT-selfadjoint} shows that the nonfractional box solution operator already possesses the spectral structure needed to define fractional powers, but only after composition with $Q$. The remaining question is whether this construction can be brought back to the primal finite-element space without introducing an arbitrary auxiliary inner product. The next definition singles out the mass-lumped inner product and the next theorem shows that it is this choice that makes the fractional box discretization intrinsic, in the sense that it is determined by the nonfractional box solution operator.

\begin{definition}\label{def:box-intrinsic}
Let $\langle\chi,\psi\rangle_{Q}=\ip{Q\chi}{Q\psi}$ be the lumped-mass inner product on $V_{h}$, and let $L_{h,Q}$ and $P_{h,Q}$ be the operators defined by \eqref{eq:operators-LP} with this inner product. The corresponding \emph{intrinsic box-method discretization} is
$u_{\beta,h}^Q= T_{h,\beta}^Qf:=(L_{h,Q})^{-\beta}P_{h,Q}f$.
\end{definition}

\begin{theorem}\label{thm:lumped-characterization}
For every $\beta>0$,
$T_{h,\beta}^Q=Q^{-1}(QT_{h})^{\beta}$,
where $T_{h}=L_{h}^{-1}P_{h}$ is the nonfractional box solution operator.
In particular, $T_{h,\beta}^Q$ depends only on $T_{h}$ and therefore, by Remark~\ref{rem:beta1-independent}, is independent of the auxiliary inner product used to represent the nonfractional problem.
\end{theorem}

\begin{proof}
We work with the lumped-mass inner product. Also, recall that $\Pi_{0,h}$ is the $L^{2}(\cS)$-orthogonal projector onto $V_{0,h}$. Since $Q: (V_{h},\langle\cdot,\cdot\rangle_{Q})\to (V_{0,h},\ip{\cdot}{\cdot})$ is an isometric isomorphism, we have $Q^*=Q^{-1}$. Hence $Q$ implements a unitary equivalence between these two Hilbert spaces, and therefore the operator
\begin{equation*}
S_{h}=Q L_{h,Q}^{-1} Q^{-1}:V_{0,h}\to V_{0,h}
\end{equation*}
is positive and selfadjoint on $V_{0,h}$. Moreover, for every $f\in L^2(\cS)$ and $\chi\in V_h$,
$$\langle P_{h,Q}f,\chi\rangle_Q=\ip{f}{Q\chi}
=\ip{\Pi_{0,h}f}{Q\chi}
=\langle Q^{-1}\Pi_{0,h}f,\chi\rangle_Q.$$
Since this holds for all $\chi\in V_h$, it follows that $P_{h,Q}=Q^{-1}\Pi_{0,h}$, and therefore
$$QT_h=Q L_{h,Q}^{-1}P_{h,Q}=Q L_{h,Q}^{-1}Q^{-1}\Pi_{0,h}=S_h^{-1}\Pi_{0,h}.$$

Recall that $Q$ implements a unitary equivalence between $L_{h,Q}^{-1}$ on $V_h$ and $S_h=Q L_{h,Q}^{-1}Q^{-1}$ on $V_{0,h}$. By the spectral calculus, fractional powers are preserved under unitary equivalence, and thus
$S_h^{-\beta}=Q L_{h,Q}^{-\beta}Q^{-1}$.
Moreover, $\Pi_{0,h}$ is the $L^2(\cS)$-orthogonal projector onto $V_{0,h}$, and $S_h^{-1}$ is positive selfadjoint on $V_{0,h}$. Thus $S_h^{-1}\Pi_{0,h}$ is precisely the extension of $S_h^{-1}$ to $L^2(\cS)$ by zero on $V_{0,h}^{\perp}$. Therefore, its fractional powers are given by the same extension of the fractional powers of $S_h^{-1}$, and hence
$$(QT_h)^\beta=(S_h^{-1}\Pi_{0,h})^\beta=S_h^{-\beta}\Pi_{0,h}=Q L_{h,Q}^{-\beta}Q^{-1}\Pi_{0,h}.$$
Applying $Q^{-1}$ and using again $P_{h,Q}=Q^{-1}\Pi_{0,h}$ yields
$Q^{-1}(QT_h)^\beta=(L_{h,Q})^{-\beta}P_{h,Q}= T_{h,\beta}^Q.$
\end{proof}

Theorem~\ref{thm:lumped-characterization} is the precise form of the structural statement announced in the introduction: among all admissible inner products, the lumped-mass choice is exactly the one that turns the box-method fractional discretization into a construction intrinsic to the nonfractional box solution operator.
We also retain the family of inner products proposed in \cite{voet2023mathematical}. 

\begin{proposition}\label{prop:family-inner-products}
Let $N_{h}=\dim V_{h}$ and let $\mathbf M$ be the mass matrix, $\mathbf M_{jk}=\ip{\phi_{j}}{\phi_{k}}$.
For $0\le \ell\le N_h-1$, write $\mathbf M=\mathbf D_{\ell}+\mathbf R_{\ell}$, where $\mathbf D_{\ell}$ contains the entries of $\mathbf M$ within bandwidth $\ell$, and $\mathbf R_{\ell}$ is the remainder. Let $\mathcal L$ denote the mass-lumping operator, defined on a matrix $\mathbf A$ by
\begin{equation*}
\mathcal L(\mathbf A)_{mn} = \delta_{mn}\sum_{k=1}^{N_h}\mathbf A_{mk},
\end{equation*}
where $\delta_{mn}$ is the Kronecker delta, and define $\mathbf M_{\ell}=\mathbf D_{\ell}+\mathcal L(\mathbf R_{\ell})$. Then
\begin{equation}\label{ind_norm}
\langle \chi,\psi\rangle_{\ell}=[\chi]^T\mathbf M_{\ell}[\psi],\qquad\chi,\psi\in V_h,
\end{equation}
defines an admissible inner product in the sense of Definition~\ref{def:admissible-innerproduct}.
\end{proposition}
\begin{proof}
Given matrices $\textbf{A},\textbf{B}\in\mathbb{R}^{N_{h}\times N_h}$, we write $\textbf{A}\preceq\textbf{B}$ (respectively, $\textbf{A}\prec\textbf{B}$) if $\textbf{B}-\textbf{A}$ is positive semidefinite (respectively, positive definite). By \cite[Lemma 3.19]{voet2023mathematical},
\begin{equation}\label{s_eq:4:60}
    -\mathcal{L}(\textbf{A}) \preceq \textbf{A} \preceq \mathcal{L}(\textbf{A})
\end{equation}
for every symmetric matrix with positive entries $\textbf{A}\in\mathbb{R}^{N_{h}\times N_h}$. Observing that $\mathcal{L}(\textbf{M})=\mathcal{L}(\textbf{M}_{\ell})=\mathcal{L}(\textbf{D}_{\ell}+\mathcal{L}(\textbf{R}_{\ell}))$, it follows from \eqref{s_eq:4:60} that
\begin{equation}\label{s_eq:4:61}
 0\prec \textbf{M}= \textbf{D}_{\ell} + \textbf{R}_{\ell} \preceq \textbf{D}_{\ell} + \mathcal{L}(\textbf{R}_{\ell}) = \textbf{M}_{\ell} \preceq \mathcal{L}(\textbf{M})
\end{equation}
for every $0\le \ell\le N_{h}-1$. Since the norms induced by $\textbf{M}$ and $\mathcal{L}(\textbf{M})$, in the sense of \eqref{ind_norm}, correspond to the $L^{2}$ inner product and the lumped-mass inner product $\langle\cdot,\cdot\rangle_{Q}$ over $V_{h}$, respectively, it follows from \eqref{s_eq:4:61} that $\|\cdot\|_{\ell}:=\sqrt{\langle\cdot,\cdot\rangle_{\ell}}$ is equivalent to $\|\cdot\|$, with bounding constants independent of $h$. Again, from \eqref{s_eq:4:61},
\begin{equation}\label{s_eq:4:62}
    0\le \|\phi\|^{2}_{\ell}-\|\phi\|^{2} \le \|\phi\|_{Q}^{2}-\|\phi\|^{2}\lesssim h^{2}|\phi|_{1}^{2}, \quad \forall \phi \in V_{h},
\end{equation}
where the last estimate holds due to \cite[Lemma 5.1]{chatzipantelidis2002}. What remains is to extend \eqref{s_eq:4:62} to the inner product $\langle\cdot,\cdot\rangle_{\ell}$. To this end, define the auxiliary form
$b(\chi,\psi):=\langle \chi,\psi \rangle_{\ell}-(\chi,\psi)$
and note that $b$ is bilinear, symmetric and, due to \eqref{s_eq:4:62}, $b(\chi,\chi)\geq 0$ for every $\chi \in V_{h}$. Hence the Cauchy--Schwarz inequality applies to $b$, and we obtain
\begin{equation*}
|\langle \chi,\psi \rangle_{\ell}-(\chi,\psi)|=|b(\chi,\psi)|\le b(\chi,\chi)^{\frac{1}{2}}b(\psi,\psi)^{\frac{1}{2}}\lesssim h^{2}|\chi|_{1}|\psi|_{1}.
\qedhere
\end{equation*}
\end{proof}

\begin{remark}
In Proposition~\ref{prop:family-inner-products}, we have $\|\cdot\|=\|\cdot\|_{N_h-1}$ and $\|\cdot\|_{0}=\|\cdot\|_{Q}$, where $\langle\cdot,\cdot\rangle_{Q}=\ip{Q\cdot}{Q\cdot}$ is the lumped-mass inner product of Definition~\ref{def:box-intrinsic}.
\end{remark}

\section{Error analysis}\label{sec:error}

In this section we add one further mesh assumption: the family $\{\cT_{h}\}$ is quasi-uniform in the sense stated at the beginning of Section~\ref{sec:setting}. This is the standard hypothesis under which the global inverse estimate $\seminorm{\chi}_{1}\lesssim h^{-1}\norm{\chi}$ and the spectral bound $\sup\sigma(L_{h})\lesssim h^{-2}$ hold uniformly. The structural results from the previous sections did not require this stronger hypothesis.

Fix any admissible inner product $\langle\cdot,\cdot\rangle_{h}$. We now compare the FEM discretization based on $\langle\cdot,\cdot\rangle_{h}$, namely \eqref{eq:fem-discrete}, and the box-method discretization based on the same inner product, namely \eqref{eq:box-general}, with the exact solution $u_{\beta}=L^{-\beta}f$. Let $\widehat L_{h}$ denote the operator induced, in the sense of \eqref{eq:operators-LP}, by the exact form $a$ and the exact $L^{2}(\cS)$ inner product on $V_{h}$, and let $\Pi_{h}$ be the $L^{2}(\cS)$-orthogonal projector onto $V_{h}$. Similarly, we write $\widetilde L_{h}$ for the operator induced by $a_{h}$ and the exact $L^{2}(\cS)$ inner product. Thus
\begin{equation}\label{expressionsss}
\widehat u_{\beta,h}=\widehat L_{h}^{-\beta}\Pi_{h}f,
\qquad
\widetilde u_{\beta,h}=\widetilde L_{h}^{-\beta}\Pi_{h}f,
\qquad
\bar u_{\beta,h}=L_{h}^{-\beta}\Lambda_{h}f.
\end{equation}
The first approximation is the classical FEM discretization; the second incorporates the quadrature error in the bilinear form while retaining the exact $L^{2}(\cS)$ projector; and the third further accounts for the change of inner product. The next lemma records the $L^{2}(\cS)$ estimate for the first of these perturbation steps; its proof is given in Section~\ref{sec:proofs}.

\begin{lemma}\label{lem:beta1-bilinear}
Assume an admissible bilinear form and let $\widehat u_{h}=\widehat L_{h}^{-1}\Pi_{h}f$ and $\widetilde u_{h}=\widetilde L_{h}^{-1}\Pi_{h}f$. Then, 
$\norm{\widehat u_{h}-\widetilde u_{h}}\lesssim h^{2}\norm{f}$ for sufficiently small $h$.
\end{lemma}

The next theorem states the rough-data estimate for the finite-element discretization under an arbitrary admissible inner product.

\begin{theorem}
\label{thm:fem-rough}

Given an admissible bilinear form and an admissible inner product.
Let $f\in L^{2}(\cS)$, let ${0\le \delta\le 1}$, and assume $2\beta>\delta$. Then, for all sufficiently small $h$,
\begin{equation}\label{eq:fem-rough}
\norm{u_{\beta}-\bar u_{\beta,h}}_{\delta}
\lesssim
h^{\min\{2\beta-\delta,2-\delta\}}
\left(\log \frac1h\right)^{[2\beta-\delta=2,\,\delta>0]}
\norm{f}.
\end{equation}
\end{theorem}

We can improve the rate from Theorem \ref{thm:fem-rough} by requiring the data to be smooth:

\begin{theorem}
\label{thm:fem-smooth}

Given an admissible bilinear form and an admissible inner product.
Let $f\in \dot H^{\sigma}$ with $0\le \sigma\le 1$, let $0\le \delta\le 1$, and assume $2\beta+\sigma>\delta$. Then, for all sufficiently small $h$,
\begin{equation}\label{eq:fem-smooth}
\norm{u_{\beta}-\bar u_{\beta,h}}_{\delta}
\lesssim
h^{\min\{2\beta+\sigma-\delta,2-\delta\}}
\left(\log \frac1h\right)^{[\{2\beta+\sigma-\delta=2\}\lor \{2\beta-\delta=2, \delta>0\}]}
\norm{f}_{\sigma}.
\end{equation}
\end{theorem}


Finally, we return to the box-method discretization. Relative to the FEM scheme, the only additional change is the replacement of $\Lambda_{h}$ by $P_{h}$, that is, the use of control-volume averaging in the load term.

\begin{theorem}
\label{thm:box-smooth}
Assume the hypotheses of Theorem~\ref{thm:fem-smooth}. Let $u_{\beta,h}=L_{h}^{-\beta}P_{h}f$ be the box-method discretization from Definition~\ref{def:box-general}. Then, for all sufficiently small $h$,
\begin{equation}\label{eq:box-smooth}
\norm{u_{\beta}-u_{\beta,h}}_{\delta}
\lesssim
h^{\min\{2\beta+\sigma-\delta, 2-\delta,1+\sigma\}}
\left(\log \frac1h\right)^{[\{2\beta+\sigma-\delta=2\}\lor \{2\beta-\delta=2, \delta>0\}]}
\norm{f}_{\sigma}.
\end{equation}
\end{theorem}


Theorem~\ref{thm:box-smooth} shows the expected $h^{1+\sigma}$ saturation caused by replacing the exact load functional $\chi\mapsto \ip{f}{\chi}$ with the control-volume functional $\chi\mapsto \ip{f}{Q\chi}$. For the purposes of this paper, the main point is that this is precisely the natural accuracy barrier of the box-method load approximation, and that the choice of admissible inner product does not change the resulting convergence order.

To pass from the discrete fractional operators analyzed above to fully implementable methods, we allow for an additional exponentially convergent rational approximation of the map $\lambda\mapsto \lambda^{-\beta}$. This covers, e.g., the rational approximations of \cite{bolin2020rational,bolin2024covariance} and the quadrature-based constructions of \cite{bonito2015numerical}.

\begin{assumption}\label{ass:rational-approx}
For each sufficiently small $h$ and each $m\in\mathbb N$, let $Q_m^\beta$ be a rational function defined on a neighborhood of $\sigma(\widetilde L_h)\cup\sigma(L_h)$, and assume that, for every $0\le \delta\le 1$,
\begin{equation}\label{eq:rational-approx}
\sup_{\lambda\in\sigma(\widetilde L_h)\cup\sigma(L_h)}
\lambda^{\delta/2}\bigl|\lambda^{-\beta}-Q_m^\beta(\lambda)\bigr|
\le
\rho_\delta(h)e^{-Cg(m)},
\end{equation}
where $C>0$ is independent of $h$ and $m$, the function $g:\mathbb N\to(0,\infty)$ satisfies $g(m)\to\infty$ as $m\to\infty$, and $\rho_\delta(h)$ is a nonnegative rational function of $h$.
\end{assumption}

We define the fully implementable approximations
\begin{equation*}
\bar u_{\beta,h}^{(m)}:=Q_m^\beta(L_h)\Lambda_h f
\qquad\text{and}\qquad
u_{\beta,h}^{(m)}:=Q_m^\beta(L_h)P_h f.
\end{equation*}

\begin{corollary}\label{cor:fully-implementable}
Suppose that $Q_m^\beta$ satisfies  Assumption~\ref{ass:rational-approx} and let $0\le \delta\le 1$. Then,
\begin{align}
\|\bar u_{\beta,h}-\bar u_{\beta,h}^{(m)}\|_{\delta,h}
\lesssim
\rho_\delta(h)e^{-Cg(m)}\|f\|,\label{eq:rational-fem-discrete} \\
\|u_{\beta,h}-u_{\beta,h}^{(m)}\|_{\delta,h}
\lesssim
\rho_\delta(h)e^{-Cg(m)}\|f\|.\label{eq:rational-box-discrete}
\end{align}
Combining these estimates with Theorem~\ref{thm:fem-rough} and Theorem~\ref{thm:box-smooth} yields
\begin{align}
\|u_\beta-\bar u_{\beta,h}^{(m)}\|_\delta
&\lesssim
h^{\min\{2\beta-\delta,2-\delta\}}
\left(\log\frac1h\right)^{[2\beta-\delta=2,\,\delta> 0]}
\|f\|
\notag\\
&\quad+
\rho_\delta(h)e^{-Cg(m)}\|f\|.\label{eq:rational-fem-total}\\
\|u_\beta-u_{\beta,h}^{(m)}\|_\delta
&\lesssim
h^{\min\{2\beta+\sigma-\delta,2-\delta,1+\sigma\}}
\left(\log\frac1h\right)^{[\{2\beta+\sigma-\delta=2\}\lor \{2\beta-\delta=2, \delta>0\}]}
\|f\|_\sigma
\notag\\
&\quad+
\rho_\delta(h)e^{-Cg(m)}\|f\|_\sigma.\label{eq:rational-box-total}
\end{align}
In particular, if $m=m(h)$ is chosen so that the exponential term is of the same order as the corresponding spatial discretization error, then the fully implementable approximation has the same asymptotic convergence rate as the exact discrete method.
\end{corollary}

\begin{proof}

Observe that $\|\bar u_{\beta,h}-\bar u_{\beta,h}^{(m)}\|_{\delta,h}= \|L_h^{\delta/2}(L_h^{-\beta}-Q_m^\beta(L_h))\Lambda_h f\|_h
$. Hence, by the spectral theorem,
$$
\|\bar u_{\beta,h}-\bar u_{\beta,h}^{(m)}\|_{\delta,h}
\le
\sup_{\lambda\in\sigma(L_h)}
\lambda^{\delta/2}\bigl|\lambda^{-\beta}-Q_m^\beta(\lambda)\bigr|
\,\|\Lambda_h f\|_h.
$$
Using Assumption~\ref{ass:rational-approx} and the uniform boundedness of $\Lambda_h$, we obtain \eqref{eq:rational-fem-discrete}. The estimate \eqref{eq:rational-box-discrete} follows in the same way, using the uniform boundedness of $P_h$. The bounds \eqref{eq:rational-fem-total} and \eqref{eq:rational-box-total} then follow immediately from the triangle inequality together with Theorems~\ref{thm:fem-rough} and \ref{thm:box-smooth}.
\end{proof}

\section{Proofs of intermediate results and error estimates}\label{sec:proofs}

This section gathers the proofs of the results from Sections~\ref{sec:setting}--\ref{sec:error} whose arguments are more technical. The notation is that of the previous sections. 
We also use the fractional Hilbert scale $\dot H^{s}(\mathcal{S})$ associated with the positive selfadjoint operator $L=-\divop(A\nabla)+\kappa^{2}$, with $\dot H^{s}(\mathcal{S})=D(L^{s/2})$ for $s\ge0$ and the completion/duality definition for $s<0$. The contour formulas used later are standard consequences of the holomorphic functional calculus for positive sectorial operators; see, for example, \cite{Haase2006,bonito2015numerical}. For Hilbert spaces $X$ and $Y$, we write $L(X,Y)$ for the space of bounded linear operators from $X$ to $Y$, and abbreviate $L(X,X)$ by $L(X)$. The corresponding operator norm is denoted by $\|\cdot\|_{L(X,Y)}$, and by $\|\cdot\|_{L(X)}$ when $X=Y$.

\subsection{Proofs from Section~\ref{sec:setting}}
For a simplex $K\in\cT_h$ and a vertex $z\in\cZ(K)$, let
$
\Gamma_{z,K}:=\partial b_z\cap K.
$
This is the part of the control-volume boundary contained in $K$. When $d=2$, $\Gamma_{z,K}$ is the broken line joining the two edge midpoints adjacent to $z$ to the barycenter of $K$. When $d=3$, $\Gamma_{z,K}$ is the union of three quadrilateral faces.

\begin{lemma}\label{supp:localflux}
Let $K\in\cT_h$, let $z\in\cZ(K)$, and let $\phi_z$ denote the local nodal basis function associated with $z$. For every constant vector $p\in\R^{d}$,
\begin{equation}\label{supp:localflux_identity}
-\int_{\Gamma_{z,K}} p\cdot\eta_{z,K}\,ds = \int_{K} p\cdot\nabla\phi_z\,dx = |K|\,p\cdot\nabla\phi_z.
\end{equation}
Here $\eta_{z,K}$ is the outward unit normal to $b_z$ restricted to $\Gamma_{z,K}$.
\end{lemma}

\begin{proof}
Both sides of \eqref{supp:localflux_identity} are covariant under affine changes of variables. Indeed, if $F(\hat x)=B\hat x+b$ maps the reference simplex $\widehat K$ to $K$, then
$\nabla\phi_z = B^{-T}\nabla\widehat\phi_z$ and 
$|K|=|\det B|\,|\widehat K|$,
and the surface integral transforms by the standard affine change-of-variables formula. Therefore it is enough to verify \eqref{supp:localflux_identity} on the reference simplex.

For $d=1$, $\widehat K=(0,1)$. For the left endpoint, $\Gamma_{0,\widehat K}=\{1/2\}$ and the outward normal equals $+1$, while $\widehat\phi_0(x)=1-x$. Hence
$$
-\int_{\Gamma_{0,\widehat K}} p\,\eta\,ds=-p
=
\int_{0}^{1} p\,\widehat\phi_0'(x)\,dx.
$$
The right endpoint is identical.

For $d=2$, take $\widehat K=\operatorname{co}\{(0,0),(1,0),(0,1)\}$. Here $\operatorname{conv}\{(0,0),(1,0),(0,1)\}$ denotes the convex hull of the three vertices, i.e., the reference triangle with corners $(0,0)$, $(1,0)$, and $(0,1)$. For the vertex $z_0=(0,0)$, the boundary $\Gamma_{z_0,\widehat K}$ consists of the two segments
$\left[(\nicefrac1{2},0),\nicefrac1{3},\nicefrac1{3})\right]$ and $
\left[(\nicefrac1{3},\nicefrac1{3}),(0,\nicefrac1{2})\right].
$
The corresponding outward normal vectors integrated along the segments add up to $(1/2,1/2)$. Since $\widehat\phi_{z_0}(x,y)=1-x-y$ and $|\widehat K|=1/2$, we obtain
$$
-\int_{\Gamma_{z_0,\widehat K}} p\cdot\eta\,ds
=
-p\cdot(1/2,1/2)
=
|\widehat K|\,p\cdot\nabla\widehat\phi_{z_0}.
$$
The remaining vertices follow by symmetry.

For $d=3$, take $\widehat K=\operatorname{conv}\{(0,0,0),(1,0,0),(0,1,0),(0,0,1)\}$. For the vertex $z_0=(0,0,0)$, the set $\Gamma_{z_0,\widehat K}$ is the union of three quadrilaterals. The area vectors of those quadrilaterals are
$\left(\nicefrac1{12},\nicefrac1{24},\nicefrac1{24}\right)$,
$\left(\nicefrac1{24},\nicefrac1{12},\nicefrac1{24}\right)$,
and 
$\left(\nicefrac1{24},\nicefrac1{24},\nicefrac1{12}\right)$,
so that
$$
\int_{\Gamma_{z_0,\widehat K}}\eta\,ds
=
\left(\frac16,\frac16,\frac16\right).
$$
Since $\widehat\phi_{z_0}(x,y,z)=1-x-y-z$ and $|\widehat K|=1/6$, we conclude that
$$
-\int_{\Gamma_{z_0,\widehat K}} p\cdot\eta\,ds
=
-p\cdot\left(\frac16,\frac16,\frac16\right)
=
|\widehat K|\,p\cdot\nabla\widehat\phi_{z_0}.
$$
The other vertices are again obtained by symmetry.
\end{proof}

The lemma implies the exactness of the box diffusion term for cellwise constant coefficients. Related nonfractional flux computations may also be found in \cite{Liang2001}.

\begin{proof}[Proof of Lemma~\ref{lem:box-piecewise-constant}]
Since $\nabla\chi$ is constant on $K$, Lemma \ref{supp:localflux} gives
$$
-\sum_{z\in\cZ(K)}\psi(z)\int_{\Gamma_{z,K}} B\nabla\chi\cdot\eta_{z,K}\,ds
=
|K|\sum_{z\in\cZ(K)} \psi(z)\,B\nabla\chi\cdot\nabla\phi_z.
$$
Because $\psi=\sum_{z\in\cZ(K)}\psi(z)\phi_z$ on $K$, the right-hand side equals
\begin{equation*}
|K|\,B\nabla\chi\cdot\nabla\psi
=
\int_{K} B\nabla\chi\cdot\nabla\psi\,dx.
\end{equation*}
Therefore, we obtain
\begin{equation}\label{supp:constantcoeff_identity}
-\sum_{z\in\cZ(K)}\psi(z)\int_{\Gamma_{z,K}} B\nabla\chi\cdot\eta_{z,K}\,ds = \int_{K} B\nabla\chi\cdot\nabla\psi\,dx.
\end{equation}
The result thus follows by summing \eqref{supp:constantcoeff_identity} over all elements and adding the reaction term.
\end{proof}

\begin{proof}[Proof of Lemma~\ref{lem:avg-coefficient}]
The definition of the bilinear form $a_h^{A_h}$ uses the piecewise constant function $A_h$ determined by the cell averages $A_K:=A_h|_{K}$ and therefore falls exactly under Lemma \ref{lem:box-piecewise-constant} on every element. Summing \eqref{supp:constantcoeff_identity} over the mesh yields
$$
-\sum_{z\in\cZ_h^{\mathrm{act}}}\psi(z)\int_{\partial b_z} A_{K}\nabla\chi\cdot\eta\,ds
=
\sum_{K\in\cT_h}\int_{K}A_{K}\nabla\chi\cdot\nabla\psi\,dx.
$$
Since $\nabla\chi|_{K}$ and $\nabla\psi|_{K}$ are constant on each element and $A_{K}$ is the mean value of $A$ on $K$,
$$
\int_{K}A_{K}\nabla\chi\cdot\nabla\psi\,dx
=
\int_{K}A\nabla\chi\cdot\nabla\psi\,dx.
$$
Adding the reaction term gives
$$
a_h^{A_h}(\chi,\psi)
=
\int_{\cS}A\nabla\chi\cdot\nabla\psi\,dx
+
(\kappa Q\chi,\kappa Q\psi),
$$
which is symmetric.
\end{proof}

\begin{proof}[Proof of Theorem~\ref{thm:box-coercive}]
By Lemma \ref{lem:avg-coefficient}, we have that
\begin{equation}\label{supp:sym_form_identity}
a_h^{A_h}(\chi,\psi)
=
\int_{\cS} A\nabla\chi\cdot\nabla\psi\,dx
+
(\kappa Q\chi,\kappa Q\psi).
\end{equation}
In particular,
\begin{equation}\label{supp:sym_coerc}
a_h^{A_h}(\chi,\chi)
\ge
\alpha \seminorm{\chi}_{1}^{2}
+
\kappa_{0}^{2}\norm{Q\chi}^{2},
\end{equation}
where $\alpha>0$ is the ellipticity constant of $A$ and $\kappa_{0}=\operatorname*{ess\,inf}_{x\in\cS}\kappa(x)>0$.

The piecewise-constant norm induced by $Q$ is uniformly equivalent to the $L^{2}$ norm on $V_h$:
\begin{equation}\label{supp:Qequiv}
c_{Q}\norm{\chi}^{2}\le \norm{Q\chi}^{2}\le C_{Q}\norm{\chi}^{2}
\qquad
\forall \chi\in V_h,
\end{equation}
with constants depending only on the shape-regularity of the mesh family; see, for example, \cite{chou2000error,Thomée2006}. Consequently,
\begin{equation}\label{supp:sym_coerc2}
a_h^{A_h}(\chi,\chi)\ge c\left(\seminorm{\chi}_{H^{1}(\cS)}^{2}+\norm{\chi}^{2}\right)
\end{equation}
for a constant $c>0$ independent of $h$.

We now compare $a_h^{A}$ and $a_h^{A_h}$. Fix $\chi,\psi\in V_h$. On each element $K$, set
$$
\bar\psi_{K}:=\frac{1}{d+1}\sum_{z\in\cZ(K)}\psi(z).
$$
As in the previous proof, let $A_K := A_h|_K$. Since the sum of the local fluxes over all control-volume interfaces in $K$ vanishes, we may subtract the constant $\bar\psi_{K}$ from the nodal values and write
\begin{align*}
a_h^{A}(\chi,\psi)-a_h^{A_h}(\chi,\psi)
=
-\sum_{K\in\cT_h}\sum_{z\in\cZ(K)}
\bigl(\psi(z)-\bar\psi_{K}\bigr)
\int_{\Gamma_{z,K}} (A-A_{K})\nabla\chi\cdot\eta_{z,K}\,ds.
\end{align*}
For affine $\psi$ on $K$, the gradient $\nabla\psi$ is constant, and
\begin{equation*}
|\psi(z)-\bar\psi_{K}|\le C h_{K} \norm{\nabla\psi}_{\mathbb{R}^{d}}
=
C h_{K}|K|^{-1/2}\norm{\nabla\psi}_{L^{2}(K)};
\end{equation*}
the same applies to $\chi$.
Moreover,
\begin{equation*}
\norm{A-A_{K}}_{L^{\infty}(K)}\le C h_{K}\norm{A}_{W^{1,\infty}(K)},
\qquad
\mathcal H^{d-1}(\Gamma_{z,K})\le C h_{K}^{d-1},
\end{equation*}
where $\mathcal H^{d-1}(\Gamma_{z,K})$ denotes the $(d-1)$-dimensional Hausdorff measure of $\Gamma_{z,K}$, that is, its length when $d=2$ and its surface area when $d=3$.
Therefore, using $|K|\simeq h_{K}^{d}$ from shape regularity,
\begin{align*}
\left|a_h^{A}(\chi,\psi)-a_h^{A_h}(\chi,\psi)\right|
&\le
C\sum_{K\in\cT_h} h_{K}^{d+1} \norm{\nabla\chi}_{\mathbb{R}^{d}}\norm{\nabla\psi}_{\mathbb{R}^{d}} \\
&\le
C h\sum_{K\in\cT_h} |K|\,\norm{\nabla\chi}_{\mathbb{R}^{d}}\norm{\nabla\psi}_{\mathbb{R}^{d}}\\
&=
C h\sum_{K\in\cT_h}\norm{\nabla\chi}_{L^{2}(K)}\norm{\nabla\psi}_{L^{2}(K)}\\
&\le
C h\,\seminorm{\chi}_{H^{1}(\cS)}\seminorm{\psi}_{H^{1}(\cS)}.
\end{align*}
Combining this perturbation bound with \eqref{supp:sym_coerc2} yields
\begin{equation*}
a_h^{A}(\chi,\chi)
\ge
a_h^{A_h}(\chi,\chi)-C h \seminorm{\chi}_{H^{1}(\cS)}^{2}
\ge
\frac{c}{2}\norm{\chi}_{H^{1}(\cS)}^{2}
\end{equation*}
for all sufficiently small $h$. This proves the theorem.
\end{proof}

\begin{proof}[Proof of Proposition~\ref{prop:pseudo-proj}]
Let $P_h$ and $\Lambda_h$ be defined by
$
\innerh{P_hf}{\chi}=(f,Q\chi)$ and 
$\innerh{\Lambda_hf}{\chi}=(f,\chi)$, respectively.
Now, for $f\in L^{2}(\cS)$ and $\chi\in V_h$,
$
\innerh{P_hf}{\chi}
=
(f,Q\chi)
=
(\Pi_{0,h}f,Q\chi)
=
\innerh{P_h\Pi_{0,h}f}{\chi}.
$
Hence $P_h=P_h\Pi_{0,h}$ which proves (i).
Further, if $\norm{Q\chi}\le C\norm{\chi}_{h}$, then
$$
|\innerh{P_hf}{\chi}|
=
|(f,Q\chi)|
\le
\norm{f}\,\norm{Q\chi}
\le
C\norm{f}\,\norm{\chi}_{h}.
$$
Choosing $\chi=P_hf$ yields $\norm{P_hf}_{h}\le C\norm{f}$, which proves (ii).
To prove (iii), we use the expansion
$$
Q\chi=\sum_{z\in\cZ_h^{\mathrm{act}}}\chi(z)\,\mathbbm 1_{b_z},
\qquad
\Pi_{0,h}f = \sum_{z\in\cZ_h^{\mathrm{act}}}\left(\Pi_{0,h}f|_{b_z}\right)\mathbbm 1_{b_z},
$$
to obtain
$$
\innerh{P_hf}{\chi}
=
(\Pi_{0,h}f,Q\chi)
=
\sum_{z\in\cZ_h^{\mathrm{act}}}\Pi_{0,h}f|_{b_z}\,\chi(z)\,|b_z|.
$$
Hence $P_hf=0$ if and only if $\Pi_{0,h}f|_{b_z}=0$ for every active vertex $z$. To show surjectivity, fix $\alpha\in V_h$ and set
$$
f_{\alpha}
:=
\sum_{z\in\cZ_h^{\mathrm{act}}}
\frac{\innerh{\alpha}{\phi_z}}{|b_z|}\mathbbm 1_{b_z}.
$$
Then for every $\chi\in V_h$,
\begin{align*}
(f_{\alpha},Q\chi)
&=
\sum_{z\in\cZ_h^{\mathrm{act}}}
\frac{\innerh{\alpha}{\phi_z}}{|b_z|}
(\mathbbm 1_{b_z},Q\chi)
=
\sum_{z\in\cZ_h^{\mathrm{act}}}
\innerh{\alpha}{\phi_z}\,\chi(z)
=
\innerh{\alpha}{\chi}.
\end{align*}
Thus $P_hf_{\alpha}=\alpha$.
Finally, to prove (iv), we have for $f,g\in L^{2}(\cS)$ that 
$
(f,P_hg)
=
\innerh{\Lambda_hf}{P_hg}
=
(Q\Lambda_hf,g).
$
Therefore we have $P_h^{*}=Q\Lambda_h$ as an operator on $L^{2}(\cS)$.
\end{proof}

\begin{proof}[Proof of Corollary~\ref{cor:projection-characterization}]
Assume first that $\innerh{\chi}{\psi}=(\chi,Q\psi)$ for all $\chi,\psi\in V_h$. Then
$$
\innerh{P_h\chi}{\psi}
=
(\chi,Q\psi)
=
\innerh{\chi}{\psi}
\qquad
\forall \chi,\psi\in V_h,
$$
so $P_h\chi=\chi$ for every $\chi\in V_h$. Since $\operatorname{Ran}(P_h)=V_h$ by Proposition~\ref{prop:pseudo-proj}, it follows that $P_h$ is a projection.

Conversely, suppose that $P_h$ is a projection. Since $\operatorname{Ran}(P_h)=V_h$, we have $P_h\chi=\chi$ for every $\chi\in V_h$. Then
$$
\innerh{\chi}{\psi}
=
\innerh{P_h\chi}{\psi}
=
(\chi,Q\psi)
\qquad
\forall \chi,\psi\in V_h.
$$
Thus the inner product must be $(\cdot,Q\cdot)$. Finally, Proposition~\ref{prop:pseudo-proj} gives
$
P_h^{*}=Q\Lambda_h,
$
whose range is contained in $V_{0,h}$. Since $\operatorname{Ran}(P_h)=V_h$, the operator $P_h$ cannot be selfadjoint on $L^{2}(\cS)$ and therefore is never an orthogonal projection.
\end{proof}

\subsection{Proofs of the error estimates}

In this section, we additionally assume quasi-uniformity of the mesh family,  as in Section~\ref{sec:error}. This is the hypothesis that allows the inverse estimate and the global spectral bound to be written in terms of the single mesh parameter $h$.

We let $\widehat L_h$ denote the operator induced by the exact bilinear form $a(\cdot,\cdot)$ relative to the exact $L^{2}$ inner product, while $\widetilde L_h$ denotes the operator induced by $a_h(\cdot,\cdot)$ relative to the exact $L^{2}$ inner product, and $L_h$ denotes the operator induced by $a_h(\cdot,\cdot)$ relative to the admissible inner product $\innerh{\cdot}{\cdot}$. Thus
$$
\widehat u_{\beta,h}=\widehat L_h^{-\beta}\Pi_h f,
\qquad
\widetilde u_{\beta,h}=\widetilde L_h^{-\beta}\Pi_h f,
\qquad
\bar u_{\beta,h}=L_h^{-\beta}\Lambda_h f.
$$

\begin{proof}[Proof of Lemma \ref{lem:beta1-bilinear}]
Set $e_h=\widehat u_h-\widetilde u_h$. Since $a(\widehat u_h,\chi)=\ip{f}{\chi}=a_h(\widetilde u_h,\chi)$ for all $\chi\in V_h$, we have
\begin{equation}\label{eq:eh-identity}
\begin{aligned}
a(e_h,\chi)
&=a(\widehat u_h-\widetilde u_h,\chi)
=a(\widehat u_h,\chi)-a(\widetilde u_h,\chi)\\
&=\ip{f}{\chi}-a(\widetilde u_h,\chi)
=a_h(\widetilde u_h,\chi)-a(\widetilde u_h,\chi)
=(a_h-a)(\widetilde u_h,\chi).
\end{aligned}
\end{equation}
Taking $\chi=e_h$ and using Definition~\ref{def:admissible-bilinear} together with the coercivity of $a$ gives
\begin{equation*}
\|e_h\|_{H^1(\cS)}^2
\lesssim a(e_h,e_h)
=(a_h-a)(\widetilde u_h,e_h)
\lesssim h^2|\widetilde u_h|_1\,|e_h|_1
\lesssim h^2\|\widetilde u_h\|_{H^1(\cS)}\,\|e_h\|_{H^1(\cS)}.
\end{equation*}
Hence
\begin{equation}\label{eq:H1-eh-unified}
\|e_h\|_{H^1(\cS)}\lesssim h^2\|\widetilde u_h\|_{H^1(\cS)}.
\end{equation}
Next, by coercivity of $a_h$ on $V_h$ for sufficiently small $h$ (see Remark~\ref{rem:ah-coercive}) and the discrete equation for $\widetilde u_h$,
\begin{equation*}
\|\widetilde u_h\|_{H^1(\cS)}^2
\lesssim a_h(\widetilde u_h,\widetilde u_h)
=\ip{f}{\widetilde u_h}
\le \|f\|\,\|\widetilde u_h\|
\le \|f\|\,\|\widetilde u_h\|_{H^1(\cS)}.
\end{equation*}
Therefore
$\|\widetilde u_h\|_{H^1(\cS)}\lesssim \|f\|$.
Combining \eqref{eq:H1-eh-unified} with this bound, we obtain
$\|e_h\|_{H^1(\cS)}\lesssim h^2\|f\|.$
Further, since $\|e_h\|\le \|e_h\|_{H^1(\cS)}$, it follows that $\|e_h\|\lesssim h^2\|f\|$.
\end{proof}

The next inverse estimate for discrete fractional norms is used repeatedly in the perturbation arguments. Its proof follows the same spectral argument as \cite[Lemma~2.2]{Jin_2023}; the only point to verify in the present setting is that the spectrum of $L_h$ remains uniformly contained in $[c,Ch^{-2}]$.

\begin{proposition}\label{prop:inverse-fractional}
Let $r,s\in\mathbb{R}$ and define 
$\norm{\chi}_{s,h}=\|L_h^{s/2}\chi\|_h$ for $\chi\in V_h$. 
Then, uniformly for sufficiently small $h$,
\begin{equation}\label{eq:inverse-fractional}
\norm{\chi}_{r,h}\lesssim h^{-(r-s)^+}\norm{\chi}_{s,h}, \qquad \chi\in V_h,
\end{equation}
where $(r-s)^+=\max\{r-s,0\}$.
\end{proposition}

\begin{proof}
It is enough to consider the case $r\ge s$. By coercivity and continuity of $a_h$ on $V_h$, together with the norm equivalence $\|\cdot\|_h\simeq \|\cdot\|$, the Rayleigh quotient of $L_h$ satisfies
\begin{equation*}
c \le \frac{\langle L_h\chi,\chi\rangle_h}{\|\chi\|_h^2}=\frac{a_h(\chi,\chi)}{\|\chi\|_h^2}\le
C\frac{|\chi|_1^2+\|\chi\|^2}{\|\chi\|^2}\lesssim h^{-2},\qquad \chi\in V_h\setminus\{0\},
\end{equation*}
where the last step is the standard inverse estimate for piecewise affine finite elements. Hence the spectrum of $L_h$ is contained in $[c,Ch^{-2}]$. Writing $\chi$ in an $\langle\cdot,\cdot\rangle_h$-orthonormal eigenbasis of $L_h$, the estimate \eqref{eq:inverse-fractional} follows  as in \cite[Lemma~2.2]{Jin_2023} by comparing the spectral weights.
\end{proof}

\begin{proof}[Proof of Theorem~\ref{thm:fem-rough}]
We decompose the error as
\begin{equation}\label{supp:Edecomp}
u_{\beta}-\bar u_{\beta,h}
=
\bigl(u_{\beta}-\widehat u_{\beta,h}\bigr)
+
\bigl(\widehat u_{\beta,h}-\widetilde u_{\beta,h}\bigr)
+
\bigl(\widetilde u_{\beta,h}-\bar u_{\beta,h}\bigr)
=:E_{1}+E_{2}+E_{3}.
\end{equation}
The term $E_{1}$ is the classical finite element error for fractional powers. Thus
\begin{equation}\label{supp:E1}
\norm{E_{1}}_{\delta}
\lesssim
h^{\min\{2\beta-\delta,2-\delta\}}
\left(\log\frac1h\right)^{[2\beta-\delta=2,\ \delta>0]}
\norm{f};
\end{equation}
see \cite{bonito2015numerical,cox2020regularity,Fujita_1991}.

We next estimate $E_{3}$, which is the most delicate term. The key observation is that
\begin{equation}\label{supp:beta1_identity}
\widetilde L_h^{-1}\Pi_h=L_h^{-1}\Lambda_h=:G_{h},
\end{equation}
because both sides are the discrete solution operator for the nonfractional problem associated with the bilinear form $a_h$. 

Let $\tilde\lambda_{1,h}$ and $\lambda_{1,h}$ denote the smallest eigenvalues of $\widetilde L_h$ and $L_h$, respectively. Since both operators are positive selfadjoint, their smallest eigenvalues admit the Rayleigh quotient characterizations
$$
\tilde\lambda_{1,h}
=
\inf_{\chi\in V_h\setminus\{0\}}
\frac{a_h(\chi,\chi)}{\|\chi\|^2},
\qquad
\lambda_{1,h}
=
\inf_{\chi\in V_h\setminus\{0\}}
\frac{a_h(\chi,\chi)}{\|\chi\|_h^2}.
$$
By Remark~\ref{rem:ah-coercive} and Definition~\ref{def:admissible-innerproduct}(ii), both quotients are bounded below by a positive constant independent of $h$. Hence there exists $c_0>0$, independent of $h$, such that
$\tilde\lambda_{1,h}\ge c_0$ and 
$\lambda_{1,h}\ge c_0$.
We therefore fix once and for all
$0<r<\nicefrac{c_0}{2},$
so that $r<\min\{\tilde\lambda_{1,h},\lambda_{1,h}\}$ for all sufficiently small $h$. In particular, $r$ is independent of $h$.

Next, let $\tilde\lambda_{N_h,h}$ and $\lambda_{N_h,h}$ denote the largest eigenvalues of $\widetilde L_h$ and $L_h$. By the global inverse estimate, there exists $C>0$, independent of $h$, such that
$
\max\{\tilde\lambda_{N_h,h},\lambda_{N_h,h}\}\le Ch^{-2}.
$
We then set
$
R=(C+1)h^{-2}.
$

We now recall the contour representation of the fractional power. Let $A$ be a finite-dimensional positive selfadjoint operator with $\sigma(A)\subset [r,R]$. For any contour $\Gamma \subseteq \mathbb{C}\setminus (-\infty,0]$ encircling $\sigma(A)$, which avoids the branch $(-\infty,0]$, the Dunford--Taylor formula gives
\begin{equation}\label{supp:frac_formula_contour}
A^{-\beta}
=
\frac{1}{2\pi i}\int_{\Gamma} z^{-\beta}(zI-A)^{-1}\,dz;
\end{equation}
see \cite{Haase2006}. We apply this formula to the keyhole contour
$$
\Gamma_{\omega}
=
\Gamma_{1,\omega}\cup\Gamma_{2,\omega}\cup\Gamma_{3,\omega}\cup\Gamma_{4,\omega},
\qquad 0<\omega<\pi,
$$
where
\begin{align*}
\Gamma_{1,\omega}&=\{z(t)=te^{i\omega}:t\in\overleftarrow{[r,R]}\},\quad
\Gamma_{2,\omega}=\{z(\theta)=Re^{i\theta}:\theta\in\overrightarrow{[-\omega,\omega]}\},\\
\Gamma_{3,\omega}&=\{z(t)=te^{-i\omega}:t\in\overrightarrow{[r,R]}\},\quad
\Gamma_{4,\omega}=\{z(\theta)=re^{i\theta}:\theta\in\overleftarrow{[-\omega,\omega]}\}.
\end{align*}
Writing the contour integral explicitly, we obtain
\begin{equation}\label{supp:frac_formula_keyhole}
\begin{split}
2\pi i A^{-\beta}
&=
e^{(1-\beta)i\omega}\int_R^r t^{-\beta}(te^{i\omega}-A)^{-1}\,dt
+
e^{(\beta-1)i\omega}\int_r^R t^{-\beta}(te^{-i\omega}-A)^{-1}\,dt\\
&\quad
+iR^{1-\beta}\int_{-\omega}^{\omega}e^{i(1-\beta)\theta}(Re^{i\theta}-A)^{-1}\,d\theta
-
ir^{1-\beta}\int_{-\omega}^{\omega}e^{i(1-\beta)\theta}(re^{i\theta}-A)^{-1}\,d\theta.
\end{split}
\end{equation}
Passing to the limit $\omega\uparrow\pi$, we arrive at
\begin{equation}\label{supp:frac_formula}
\begin{split}
A^{-\beta}
&=
\frac{\sin(\pi\beta)}{\pi}\int_r^R t^{-\beta}(t+A)^{-1}\,dt
+
\frac{R^{1-\beta}}{2\pi}\int_{-\pi}^{\pi}e^{i(1-\beta)\theta}(Re^{i\theta}-A)^{-1}\,d\theta\\
&\quad-
\frac{r^{1-\beta}}{2\pi}\int_{-\pi}^{\pi}e^{i(1-\beta)\theta}(re^{i\theta}-A)^{-1}\,d\theta.
\end{split}
\end{equation}
This formula is valid in particular for $A=\widetilde L_h$ and $A=L_h$.

We now explain how \eqref{supp:frac_formula} leads to an estimate for $E_3$. Define
$$
\mathcal E_h
:=
\widetilde L_h^{\delta/2}\bigl(\widetilde L_h^{-\beta}\Pi_h-L_h^{-\beta}\Lambda_h\bigr)
\in L(L^2(\cS)).
$$
Since $E_3=(\widetilde L_h^{-\beta}\Pi_h-L_h^{-\beta}\Lambda_h)f$ and, for $0\le\delta\le1$, the norms $\|\cdot\|_\delta$ and $\|\widetilde L_h^{\delta/2}\cdot\|$ are equivalent on $V_h$ with constants independent of $h$, we have
$
\norm{E_3}_\delta
\lesssim
\|\mathcal E_h f\|
\le
\|\mathcal E_h\|_{L(L^2(\cS))}\,\|f\|.
$
Applying \eqref{supp:frac_formula} to $\widetilde L_h^{-\beta}$ and $L_h^{-\beta}$, we therefore obtain
\begin{equation}\label{supp:E3_resolvent}
\begin{split}
\norm{E_3}_\delta
\lesssim
\|f\|\Bigg(
&\int_r^R t^{-\beta}\bigl\|\widetilde L_h^{\delta/2}\mathcal I_1\bigr\|_{L(L^2(\cS))}\,dt
+
R^{1-\beta}\int_{-\pi}^{\pi}\bigl\|\widetilde L_h^{\delta/2}\mathcal I_2\bigr\|_{L(L^2(\cS))}\,d\theta\\
&\quad+
r^{1-\beta}\int_{-\pi}^{\pi}\bigl\|\widetilde L_h^{\delta/2}\mathcal I_3\bigr\|_{L(L^2(\cS))}\,d\theta
\Bigg),
\end{split}
\end{equation}
where
$\mathcal I_1=(t+\widetilde L_h)^{-1}\Pi_h-(t+L_h)^{-1}\Lambda_h$,
$\mathcal I_2=(Re^{i\theta}-\widetilde L_h)^{-1}\Pi_h-(Re^{i\theta}-L_h)^{-1}\Lambda_h$, and 
$\mathcal I_3=(re^{i\theta}-\widetilde L_h)^{-1}\Pi_h-(re^{i\theta}-L_h)^{-1}\Lambda_h$.
We first estimate $\mathcal I_1$. Using \eqref{supp:beta1_identity}, we obtain
\begin{align}
\mathcal I_1
&=
(t+\widetilde L_h)^{-1}\widetilde L_hG_h-(t+L_h)^{-1}L_hG_h 
=
t(t+\widetilde L_h)^{-1}\widetilde L_h\bigl(\widetilde L_h^{-1}-L_h^{-1}\bigr)L_h(t+L_h)^{-1}G_h \notag\\
&=
t(t+\widetilde L_h)^{-1}\widetilde L_h^{1/2}
\widetilde L_h^{1/2}\bigl(\widetilde L_h^{-1}-L_h^{-1}\bigr)
L_h(t+L_h)^{-1}G_h.
\label{supp:I1}
\end{align}
We claim that
\begin{equation}\label{supp:inner_diff_est}
\bigl\|\widetilde L_h^{1/2}\bigl(\widetilde L_h^{-1}-L_h^{-1}\bigr)\bigr\|_{L(V_h)}
\lesssim
h^{2}\bigl\|L_h^{1/2}\bigr\|_{L(V_h)},
\end{equation}
that is,
\begin{equation}\label{supp:inner_diff_est_pointwise}
\norm{\widetilde L_h^{1/2}\bigl(\widetilde L_h^{-1}-L_h^{-1}\bigr)\psi}
\lesssim
h^{2}\norm{L_h^{1/2}\psi}_{h},
\qquad
\psi\in V_h.
\end{equation}
Indeed, setting
$
\xi=(\widetilde L_h^{-1}-L_h^{-1})\psi,
$
we have
\begin{align*}
\norm{\widetilde L_h^{\nicefrac1{2}}\xi}^{2}
&=
a_h(\xi,\xi)
=
(\psi,\xi)-\innerh{\psi}{\xi}\\
&\lesssim
h^{2}\seminorm{\psi}_{1}\seminorm{\xi}_{1}
\lesssim h^2 a_h(\psi,\psi)^{\nicefrac1{2}} a_h(\xi,\xi)^{\nicefrac1{2}}\\
&= h^{2}\norm{L_h^{\nicefrac1{2}}\psi}_{h}\norm{\widetilde L_h^{\nicefrac1{2}}\xi},
\end{align*}
which proves \eqref{supp:inner_diff_est_pointwise}.
Combining \eqref{supp:I1} and \eqref{supp:inner_diff_est_pointwise}, we obtain
\begin{equation}\label{supp:I1_norm}
\bigl\|\widetilde L_h^{\delta/2}\mathcal I_1\bigr\|_{L(L^2(\cS))}
\lesssim
h^{2}t
\bigl\|\widetilde L_h^{\delta/2}(t+\widetilde L_h)^{-1}\widetilde L_h^{1/2}\bigr\|_{L(L^2(\cS))}
\,
\bigl\|L_h^{3/2}(t+L_h)^{-1}G_h\bigr\|_{L(L^2(\cS))}.
\end{equation}
We now estimate the two factors by the spectral theorem. Since $\widetilde L_h$ is positive selfadjoint,
\begin{equation}\label{supp:spec1}
\bigl\|\widetilde L_h^{\delta/2}(t+\widetilde L_h)^{-1}\widetilde L_h^{1/2}\bigr\|_{L(L^2(\cS))}
=
\sup_{\lambda\in\sigma(\widetilde L_h)}
\frac{\lambda^{\frac12+\frac\delta2}}{t+\lambda}
\lesssim
t^{\frac\delta2-\frac12}.
\end{equation}
Likewise, using $G_h=L_h^{-1}\Lambda_h$ and the uniform boundedness of $\Lambda_h$,
\begin{equation}\label{supp:spec2}
\bigl\|L_h^{3/2}(t+L_h)^{-1}T_h\bigr\|_{L(L^2(\cS))}
=
\bigl\|L_h^{1/2}(t+L_h)^{-1}\Lambda_h\bigr\|_{L(L^2(\cS))} \lesssim
\sup_{\lambda\in\sigma(L_h)}
\frac{\lambda^{1/2}}{t+\lambda}
\lesssim
t^{-1/2}.
\end{equation}
Substituting \eqref{supp:spec1} and \eqref{supp:spec2} into \eqref{supp:I1_norm} yields
$
\bigl\|\widetilde L_h^{\delta/2}\mathcal I_1\bigr\|_{L(L^2(\cS))}
\lesssim
h^2 t^{\delta/2}.
$
Therefore
\begin{equation}\label{supp:I1_final}
\int_r^R t^{-\beta}\bigl\|\widetilde L_h^{\delta/2}\mathcal I_1\bigr\|_{L(L^2(\cS))}\,dt
\lesssim
h^2\int_r^R t^{-\beta+\delta/2}\,dt
\lesssim
h^{\min\{2\beta-\delta,2\}}
\left(\log\frac1h\right)^{[2\beta-\delta=2]}.
\end{equation}
Indeed, if $2\beta-\delta\neq 2$, the integral is bounded by $R^{1-\beta+\delta/2}$, and since $R=(C+1) h^{-2}$ this gives $h^{2\beta-\delta}$ when $2\beta-\delta<2$ and $h^2$ otherwise. In the critical case $2\beta-\delta=2$, the integral equals $\log(R/r)\lesssim \log(1/h)$.

To estimate the contribution on the large circular arc, we argue as in the treatment of $\mathcal I_1$. First,
\begin{equation}\label{supp:I2_norm}
  \begin{aligned}
\bigl\|\widetilde L_h^{\delta/2}\mathcal I_2\bigr\|_{L(L^2(\cS))}
\lesssim
h^2 R
\bigl\|\widetilde L_h^{\delta/2}(Re^{i\theta}&-\widetilde L_h)^{-1}\widetilde L_h^{1/2}\bigr\|_{L(L^2(\cS))}
\,\bigl\|L_h^{3/2}(Re^{i\theta}-L_h)^{-1}G_h\bigr\|_{L(L^2(\cS))}.
  \end{aligned}
\end{equation}
We now estimate the two factors on the right-hand side by the spectral theorem. Since $\sigma(\widetilde L_h)\subset [r,Ch^{-2}]$, we have
\begin{equation}\label{supp:I2_spec1}
\begin{gathered}
R^{1-\beta / 2}\left\|\widetilde{L}_h^{\delta / 2}\left(R e^{i \theta}-\widetilde{L}_h\right)^{-1} \widetilde{L}_h^{1 / 2}\right\|_{L\left(L^2(\mathscr{S})\right)}=R^{1-\beta / 2} \sup _{\lambda \in \sigma\left(\widetilde{L}_h\right)} \frac{\lambda^{\frac{1}{2}+\frac{\delta}{2}}}{\left|R e^{i \theta}-\lambda\right|} \\
\leq R^{1-\beta / 2} \sup _{\lambda \in\left[r, C h^{-2}\right]} \frac{\lambda^{\frac{1}{2}+\frac{\delta}{2}}}{\left|R e^{i \theta}-\lambda\right|} \leq C^{\frac{1}{2}+\frac{\delta}{2}}(C+1)^{1-\frac{\beta}{2}} h^{\beta-1-\delta}
\end{gathered}
\end{equation}



Similarly, using $G_h=L_h^{-1}\Lambda_h$ and the uniform boundedness of $\Lambda_h$, we obtain
\begin{align*}
R^{1-\beta/2}
&\bigl\|L_h^{3/2}(Re^{i\theta}-L_h)^{-1}G_h\bigr\|_{L(L^2(\cS))}
=
R^{1-\beta/2}
\bigl\|L_h^{1/2}(Re^{i\theta}-L_h)^{-1}\Lambda_h\bigr\|_{L(L^2(\cS))}
\lesssim
h^{\beta-1}.
\end{align*}
Therefore, by \eqref{supp:I2_norm}, \eqref{supp:I2_spec1}, and the preceding estimate,
$$
R^{1-\beta}\bigl\|\widetilde L_h^{\delta/2}\mathcal I_2\bigr\|_{L(L^2(\cS))}
\lesssim
h^2\,
h^{\beta-\delta-1}\,
h^{\beta-1}
=
h^{2\beta-\delta}.
$$
Since the bound is uniform in $\theta$, integration over $[-\pi,\pi]$ gives
\begin{equation}\label{supp:I2_final}
R^{1-\beta}\int_{-\pi}^{\pi}\bigl\|\widetilde L_h^{\delta/2}\mathcal I_2\bigr\|_{L(L^2(\cS))}\,d\theta
\lesssim
h^{2\beta-\delta}.
\end{equation}

Finally, we consider the small circular arc. Since $r$ is fixed independently of $h$ and lies strictly below the spectra of both $\widetilde L_h$ and $L_h$, the resolvents $(re^{i\theta}-\widetilde L_h)^{-1}$ and $(re^{i\theta}-L_h)^{-1}$ are uniformly bounded in $L(V_h)$ for $\theta\in[-\pi,\pi]$. Therefore, arguing exactly as for $\mathcal I_2$, but now with $r$ fixed instead of $R\approx h^{-2}$, we obtain
\begin{equation}\label{supp:I3_final}
r^{1-\beta}\int_{-\pi}^{\pi}\bigl\|\widetilde L_h^{\delta/2}\mathcal I_3\bigr\|_{L(L^2(\cS))}\,d\theta
\lesssim
h^2.
\end{equation}
Indeed, in this case all resolvent factors remain uniformly bounded, and the only $h$-dependence comes from the factor $h^2$ arising from \eqref{supp:inner_diff_est_pointwise}.

Substituting \eqref{supp:I1_final}, \eqref{supp:I2_final}, and \eqref{supp:I3_final} into \eqref{supp:E3_resolvent}, we conclude that
\begin{equation}\label{supp:E3_final}
\norm{E_3}_{\delta}
\lesssim
h^{\min\{2\beta-\delta,2\}}
\left(\log\frac1h\right)^{[2\beta-\delta=2]}
\norm{f}.
\end{equation}

The estimate for $E_2$ is obtained by the same contour argument, now replacing the pair $(\widetilde L_h,L_h)$ by $(\widehat L_h,\widetilde L_h)$. Since both operators are realized in the exact $L^2$ inner product, no projection mismatch occurs, and the role of \eqref{supp:inner_diff_est_pointwise} is played by Lemma~\ref{lem:beta1-bilinear}. Thus
\begin{equation}\label{supp:E2_final}
\norm{E_2}_{\delta}
\lesssim
h^{\min\{2\beta-\delta,2\}}
\left(\log\frac1h\right)^{[2\beta-\delta=2]}
\norm{f}.
\end{equation}

Combining \eqref{supp:E1}, \eqref{supp:E2_final}, and \eqref{supp:E3_final} in \eqref{supp:Edecomp}, we obtain
\begin{equation*}
\norm{u_{\beta}-\bar u_{\beta,h}}_{\delta}
\lesssim
h^{\min\{2\beta-\delta,2-\delta\}}
\left(\log\frac1h\right)^{[2\beta-\delta=2,\ \delta>0]}
\norm{f}.
\end{equation*}
Finally, in the case $\delta=0$ and $2\beta=2$, one has $\beta=1$. Then the logarithmic factor disappears, because both perturbation terms are of order $h^2$ by Lemma~\ref{lem:beta1-bilinear}. This proves the theorem.
\end{proof}

\begin{proof}[Proof of Theorem~\ref{thm:fem-smooth}]
Following the notation of the proof of Theorem~\ref{thm:fem-rough}, it is enough to estimate the analogue of $E_3$. Let $g=L^{\sigma/2}f\in L^2(\cS).$
Then
$f=L^{-\sigma/2}g, \norm{g}=\norm{f}_{\sigma},$
and therefore
$$
E_3^{(\sigma)}
:=
\left\|
\bigl(\widetilde L_h^{-\beta}\Pi_h-L_h^{-\beta}\Lambda_h\bigr)L^{-\sigma/2}g
\right\|_{\delta}.
$$

We add and subtract $\bigl(\widetilde L_h^{-\beta}\Pi_h-L_h^{-\beta}\Lambda_h\bigr)\widetilde L_h^{-\sigma/2}\Pi_h g$ and use the triangle inequality:
\begin{equation}\label{supp:smooth_split}
\begin{split}
E_3^{(\sigma)}
&\le
\left\|
\bigl(\widetilde L_h^{-\beta}\Pi_h-L_h^{-\beta}\Lambda_h\bigr)
\bigl(L^{-\sigma/2}-\widetilde L_h^{-\sigma/2}\Pi_h\bigr)g
\right\|_{\delta}+
\left\|
\bigl(\widetilde L_h^{-\beta}\Pi_h-L_h^{-\beta}\Lambda_h\bigr)
\widetilde L_h^{-\sigma/2}\Pi_h g
\right\|_{\delta}.
\end{split}
\end{equation}

We begin with the first term. By the estimate \eqref{supp:E3_final}, applied with data
\begin{equation*}
\bigl(L^{-\sigma/2}-\widetilde L_h^{-\sigma/2}\Pi_h\bigr)g \in L^{2}(\cS),
\end{equation*}
we obtain
\begin{equation}\label{supp:smooth_first_term}
\begin{split}
&\left\|
\bigl(\widetilde L_h^{-\beta}\Pi_h-L_h^{-\beta}\Lambda_h\bigr)
\bigl(L^{-\sigma/2}-\widetilde L_h^{-\sigma/2}\Pi_h\bigr)g
\right\|_{\delta}\\
&\qquad\lesssim
h^{\min\{2\beta-\delta,2\}}
\left(\log\frac1h\right)^{[2\beta-\delta=2;\delta>0]}
\left\|
\bigl(L^{-\sigma/2}-\widetilde L_h^{-\sigma/2}\Pi_h\bigr)g
\right\|.
\end{split}
\end{equation}
Now, the classical finite element estimate for fractional powers gives
$$
\left\|
\bigl(L^{-\sigma/2}-\widetilde L_h^{-\sigma/2}\Pi_h\bigr)g
\right\|
\lesssim
h^{\sigma}\norm{g}.
$$
Substituting this into \eqref{supp:smooth_first_term}, we find
\begin{equation}\label{supp:smooth_first_term_final}
\begin{split}
&\left\|
\bigl(\widetilde L_h^{-\beta}\Pi_h-L_h^{-\beta}\Lambda_h\bigr)
\bigl(L^{-\sigma/2}-\widetilde L_h^{-\sigma/2}\Pi_h\bigr)g
\right\|_{\delta}\\
&\qquad\lesssim
h^{\min\{2\beta-\delta,2\}+\sigma}
\left(\log\frac1h\right)^{[2\beta-\delta=2, \delta>0]}
\norm{g}.
\end{split}
\end{equation}

We now turn to the second term in \eqref{supp:smooth_split}. As in the proof of Theorem~\ref{thm:fem-rough}, we use the contour representation of the fractional power. Thus it suffices to estimate
\begin{equation*}
\begin{split}
&\left\|
L_h^{\delta/2}
\bigl(\widetilde L_h^{-\beta}\Pi_h-L_h^{-\beta}\Lambda_h\bigr)
\widetilde L_h^{-\sigma/2}\Pi_h
\right\|_{L(L^2(\cS))}\\
&\qquad=\left\|
L_h^{\delta/2}
\bigl(\widetilde L_h^{-\beta}\Lambda_h-L_h^{-\beta}\Pi_h\bigr)
\widetilde L_h^{-\sigma/2}\Pi_h
\right\|_{L(L^2(\cS))}.
\end{split}
\end{equation*}
Proceeding exactly as before, the analogue of \eqref{supp:I1} now gives

\begin{equation*}\label{supp:smooth_I1_operator} \begin{split} 
\mathcal I_1 &= t(t+L_h)^{-1}L_h \bigl(L_h^{-1}-\widetilde L_h^{-1}\bigr) \widetilde L_h(t+\widetilde L_h)^{-1}G_h L_h^{-\sigma/2}\\ 
&= t(t+L_h)^{-1}L_h^{1/2} L_h^{1/2} \bigl(L_h^{-1}-\widetilde L_h^{-1}\bigr) \widetilde L_h(t+\widetilde L_h)^{-1}G_h L_h^{-\sigma/2}. \end{split} \end{equation*}

Using \eqref{supp:inner_diff_est_pointwise}, we obtain
\begin{equation}\label{supp:smooth_I1}
\begin{split}
\bigl\|L_h^{\delta/2}\mathcal I_1\bigr\|_{L(L^2(\cS))}
&\lesssim
h^2 t
\bigl\|L_h^{\delta/2}(t+L_h)^{-1}L_h^{1/2}\bigr\|_{L(L^2(\cS))}\cdot
\bigl\|\widetilde L_h^{3/2}(t+\widetilde L_h)^{-1}G_h\widetilde L_h^{-\sigma/2}\bigr\|_{L(L^2(\cS))}.
\end{split}
\end{equation}
By the spectral theorem,
\begin{equation}\label{supp:smooth_spec1}
\bigl\|L_h^{\delta/2}(t+L_h)^{-1}L_h^{1/2}\bigr\|_{L(L^2(\cS))}
=
\sup_{\lambda\in\sigma(L_h)}
\frac{\lambda^{\frac12+\frac\delta2}}{t+\lambda}
\lesssim
t^{\frac\delta2-\frac12},
\end{equation}
and similarly
\begin{equation}\label{supp:smooth_spec2}
\begin{split}
\bigl\|\widetilde L_h^{\frac3{2}}(t+\widetilde L_h)^{-1}G_h\widetilde L_h^{-\frac{\sigma}{2}}\bigr\|_{L(L^2(\cS))}
&=
\bigl\|\widetilde L_h^{\frac1{2}}(t+\widetilde L_h)^{-1}\Pi_h\widetilde L_h^{-\frac{\sigma}{2}}\bigr\|_{L(L^2(\cS))}\lesssim
\sup_{\lambda\in\sigma(\widetilde L_h)}
\frac{\lambda^{\frac12-\frac\sigma2}}{t+\lambda}
\lesssim
t^{-\frac12-\frac\sigma2}.
\end{split}
\end{equation}
Substituting \eqref{supp:smooth_spec1} and \eqref{supp:smooth_spec2} into \eqref{supp:smooth_I1}, we get
$
\bigl\|L_h^{\delta/2}\mathcal I_1\bigr\|_{L(L^2(\cS))}
\lesssim
h^2 t^{\frac{\delta-\sigma}{2}}.
$
Therefore
\begin{equation}\label{supp:smooth_I1_final}
\begin{split}
\int_r^R t^{-\beta}\bigl\|L_h^{\delta/2}\mathcal I_1\bigr\|_{L(L^2(\cS))}\,dt
&\lesssim
h^2\int_r^R t^{-\beta+\frac{\delta-\sigma}{2}}\,dt\lesssim
h^{\min\{2\beta+\sigma-\delta,2\}}
\left(\log\frac1h\right)^{[2\beta+\sigma-\delta=2]}.
\end{split}
\end{equation}

We next estimate the contribution on the large circular arc. As in the proof of Theorem~\ref{thm:fem-rough}, the corresponding resolvent term satisfies
\begin{equation}\label{supp:smooth_I2}
\begin{split}
\bigl\|L_h^{\delta/2}\mathcal I_2\bigr\|_{L(L^2(\cS))}
&\lesssim
h^2 R
\bigl\|L_h^{\delta/2}(Re^{i\theta}-L_h)^{-1}L_h^{1/2}\bigr\|_{L(L^2(\cS))}\\
&\qquad\cdot
\bigl\|\widetilde L_h^{3/2}(Re^{i\theta}-\widetilde L_h)^{-1}G_h\widetilde L_h^{-\sigma/2}\bigr\|_{L(L^2(\cS))}.
\end{split}
\end{equation}
Using again the spectral theorem and the fact that $R=(C+1)h^{-2}$ lies a fixed relative distance away from the spectra, we obtain
\begin{align*}
R^{1-\beta/2}
\bigl\|L_h^{\delta/2}(Re^{i\theta}-L_h)^{-1}L_h^{1/2}\bigr\|_{L(L^2(\cS))}
&\lesssim
h^{\beta-\delta-1},\\
R^{1-\beta/2}
\bigl\|\widetilde L_h^{3/2}(Re^{i\theta}-\widetilde L_h)^{-1}G_h\widetilde L_h^{-\sigma/2}\bigr\|_{L(L^2(\cS))}
&\lesssim
h^{\beta+\sigma-1}.
\end{align*}
Therefore
\begin{equation}\label{supp:smooth_I2_final}
R^{1-\beta}\int_{-\pi}^{\pi}\bigl\|L_h^{\delta/2}\mathcal I_2\bigr\|_{L(L^2(\cS))}\,d\theta
\lesssim
h^{2\beta+\sigma-\delta}.
\end{equation}

Finally, on the small circular arc, the argument is simpler because $r$ is fixed independently of $h$. As before, all resolvent factors remain uniformly bounded, and the only $h$-dependence comes from the factor $h^2$ in \eqref{supp:inner_diff_est_pointwise}. Thus
\begin{equation}\label{supp:smooth_I3_final}
r^{1-\beta}\int_{-\pi}^{\pi}\bigl\|L_h^{\delta/2}\mathcal I_3\bigr\|_{L(L^2(\cS))}\,d\theta
\lesssim
h^2.
\end{equation}

Combining \eqref{supp:smooth_I1_final}, \eqref{supp:smooth_I2_final}, and \eqref{supp:smooth_I3_final}, we conclude that
\begin{equation}\label{supp:smooth_E3_final}
\left\|
L_h^{\delta/2}
\bigl(\widetilde L_h^{-\beta}\Pi_h-L_h^{-\beta}\Lambda_h\bigr)
\widetilde L_h^{-\sigma/2}\Pi_h g
\right\|_{L(L^2(\cS))}\lesssim
h^{\min\{2\beta+\sigma-\delta,2\}}
\left(\log\frac1h\right)^{[2\beta+\sigma-\delta=2]}
\norm{g}.
\end{equation}
Together with \eqref{supp:smooth_first_term_final}, this yields
$$
E_3^{(\sigma)}
\lesssim
h^{\min\{2\beta+\sigma-\delta,2\}}
\left(\log\frac1h\right)^{[\{2\beta+\sigma-\delta=2\}\,\,\lor\,\,\{2\beta-\delta=2, \delta>0\}]}
\norm{g}.
$$

Since $\norm{g}=\norm{f}_{\sigma}$, the same bound holds with $\norm{f}_{\sigma}$ on the right-hand side. The term analogous to $E_2$ is treated exactly as in the proof of Theorem~\ref{thm:fem-rough}, replacing the pair $(\widetilde L_h,L_h)$ by $(\widehat L_h,\widetilde L_h)$. Since both operators are realized in the exact $L^2$ inner product, no projection mismatch occurs, and the resulting estimate is no worse than the one above. The classical FEM estimate for smoother data gives the result for the term analogous to $E_1$. The proof is concluded by combining the three estimates.
\end{proof}
\begin{lemma}
\label{lem:load-quadrature}
Let $0\le \sigma\le 1$. Then, for every $f\in \dot H^{\sigma}$ and every $\chi\in V_{h}$,
\begin{equation}\label{eq:load-quadrature}
\lvert \ip{f}{\chi-Q\chi}\rvert \lesssim h^{1+\sigma}\norm{f}_{\sigma}\seminorm{\chi}_{1}.
\end{equation}
\end{lemma}
\begin{proof}
For $\sigma=0$, the estimate follows from $\norm{\chi-Q\chi}\lesssim h\seminorm{\chi}_{1}$. For $\sigma=1$, it is the standard barycentric lumped-mass quadrature estimate on shape-regular simplicial meshes,
\begin{equation*}
\lvert \ip{f}{\chi-Q\chi}\rvert\lesssim h^{2}\seminorm{f}_{1}\seminorm{\chi}_{1};
\end{equation*}
see, for example, \cite[Chapter~4]{Thomée2006}. Interpolation between the endpoint estimates yields \eqref{eq:load-quadrature} for all $0\le \sigma\le 1$.
\end{proof}

\begin{proof}[Proof of Theorem~\ref{thm:box-smooth}]
By Theorem~\ref{thm:fem-smooth}, it is enough to estimate
$
\|u_{\beta,h}-\bar u_{\beta,h}\|_{\delta}.
$
Set
$
d_h:=L_h^{-\beta+\delta/2}(P_h-\Lambda_h)f.
$
Since, on $V_h$, the norm $\|\cdot\|_{\delta}$ is equivalent to the graph norm induced by $L_h^{\delta/2}$, with constants independent of $h$, it is enough to estimate $\|d_h\|$.

We first consider the case $\sigma=0$. Let $g\in L^2(\cS)$ be arbitrary. Since $\Lambda_h g\in V_h$ and
$
\langle \Lambda_h g,\chi\rangle_h=(g,\chi)$ for $\chi\in V_h$ 
we have
$
(d_h,g)=\langle d_h,\Lambda_h g\rangle_h.
$
Using the selfadjointness of $L_h$ with respect to $\langle\cdot,\cdot\rangle_h$, we obtain
$$
(d_h,g)
=
\langle (P_h-\Lambda_h)f,\,L_h^{-\beta+\delta/2}\Lambda_h g\rangle_h.
$$
By the definitions of $P_h$ and $\Lambda_h$,
$
\langle P_h f,\chi\rangle_h=(f,Q\chi)$ and 
$\langle \Lambda_h f,\chi\rangle_h=(f,\chi)$ for $\chi\in V_h$.
Therefore
$$
(d_h,g)
=
(f,Q L_h^{-\beta+\delta/2}\Lambda_h g)
-
(f,L_h^{-\beta+\delta/2}\Lambda_h g)
=
(f,(Q-I)L_h^{-\beta+\delta/2}\Lambda_h g).
$$
Hence
$
|(d_h,g)|
\le
\|f\|\,
\|(Q-I)L_h^{-\beta+\nicefrac{\delta}{2}}\Lambda_h g\|$.
Taking the supremum over all $g\in L^2(\cS)$ with $\|g\|=1$, we conclude that
\begin{equation}\label{supp:box-rough-operator}
\|d_h\|
\lesssim
\|(Q-I)L_h^{-\beta+\delta/2}\Lambda_h\|_{L(L^2(\cS))}\,\|f\|.
\end{equation}
Since $Q$ is the nodal interpolation onto dual-cell constants, we have
$
\|(I-Q)\chi\|\lesssim h\|\chi\|_{H^1(\cS)}$ for $\chi\in V_h$.
Therefore,
\begin{equation}\label{eq:d_h}
\|d_h\|
\lesssim
h\,\|L_h^{-\beta+\delta/2}\Lambda_h\|_{L(L^2(\cS);H^1(\cS))}\,\|f\|.
\end{equation}
Observe that since
$
\|L_h^{-\beta+\nicefrac{\delta}{2}}\Lambda_h g\|_{H^1(\cS)}
\lesssim
\|L_h^{-\beta+\nicefrac{\delta}{2}}\Lambda_h g\|_{1,h},
$
it is enough to estimate the discrete norm on the right. Applying Proposition~\ref{prop:inverse-fractional} with
$r=-2\beta+\delta+1$ and $s=0$,
to $\chi=\Lambda_h g$, we obtain
\begin{align*}
\|L_h^{-\beta+\delta/2}\Lambda_h g\|_{1,h}=\|\Lambda_h g\|_{-2\beta+\delta+1,h}
&\lesssim
h^{-(-2\beta+\delta+1)^+}
\|\Lambda_h g\|
\lesssim h^{\min\{2\beta-\delta,1\}-1}
\|\Lambda_h g\|.
\end{align*}
where we used in the last inequality that for every $a,b\in\mathbb{R}$ we have that $(a+b)^{+}=b-\min\{-a,b\}$. Plugging the last inequality into \eqref{eq:d_h}, we obtain
$
\|d_h\|
\lesssim
h^{\min\{2\beta-\delta,1\}}\|f\|.
$

We now consider the case $0<\sigma\le 1$. Let $g\in L^2(\cS)$ and let $w\in\cV$ solve the adjoint problem
$
Lw=g.
$
By convexity of $\cS$ and elliptic regularity, we have $w\in H^2(\cS)$ and
$
\|w\|_{H^2(\cS)}\lesssim \|g\|.
$
Let $w_h=R_hw\in V_h$ be the Rayleigh-Ritz projection of $w$, that is,
$a(w_h,\chi)=a(w,\chi)$ for $\chi\in V_h$.
We first record the estimate
\begin{equation}\label{supp:box-ritz}
\|w_h\|_{2,h}\lesssim \|w\|_{H^2(\cS)}+h\|w_h\|_{H^1(\cS)}\lesssim \|g\|,
\end{equation}
for $h\le 1$. Indeed, for every $\chi\in V_h$,
$
\langle L_h w_h,\chi\rangle_h
=
a_h(w_h,\chi)
=
a(w,\chi)-(a-a_h)(w_h,\chi).
$
Hence
$$
|\langle L_h w_h,\chi\rangle_h|
\lesssim
\|Lw\|\,\|\chi\|+h^2\|w_h\|_{H^1(\cS)}\|\chi\|_{H^1(\cS)}.
$$
Taking the supremum over $\chi\in V_h$, and using the inverse inequality together with the $H^1$-stability of the Ritz projection, gives the first bound in \eqref{supp:box-ritz}. The second follows from elliptic regularity and $H^1$-stability of the Ritz projection, by the estimate
$$
\|w_h\|_{H^1(\cS)}\lesssim \|w\|_{H^1(\cS)}\lesssim \|w\|_{H^2(\cS)}\lesssim \|g\|.
$$

We now estimate $d_h$ by duality. Using the definition of $d_h$ and the fact that $Lw=g$, we compute
\begin{equation}\label{supp:box-duality}
\begin{split}
(g,d_h)
&=
a(d_h,w)
=
a(d_h,w_h)-a_h(d_h,w_h)+a_h(d_h,w_h)\\
&=
a(d_h,w_h)-a_h(d_h,w_h)
+
\left(f,(Q-I)L_h^{-\beta+\delta/2+1}w_h\right).
\end{split}
\end{equation}

We estimate the two terms separately. For the first one, Definition~\ref{def:admissible-bilinear} gives
$$
|a(d_h,w_h)-a_h(d_h,w_h)|
\lesssim
h^2 |d_h|_1 |w_h|_1.
$$
Repeating the duality argument from the case \(\sigma=0\), now for
$L_h^{1/2}d_h
=
L_h^{-\beta+(\delta+1)/2}(P_h-\Lambda_h)f,$
and using the equivalence of the discrete and continuous energy norms, we obtain
$$
|d_h|_1
\lesssim
h^{\min\{2\beta-\delta-1,1\}}\|f\|
\lesssim
h^{\min\{2\beta-\delta-1,1\}}\|f\|_\sigma.
$$
Therefore
\begin{equation}\label{supp:box-first-term-final}
\begin{split}
|a(d_h,w_h)-a_h(d_h,w_h)|
&\lesssim
h^{2}h^{\min\{2\beta-\delta-1,1\}}\|f\|\,\|w_h\|_{H^1(\cS)}\lesssim
h^{\min\{2\beta-\delta+1,3\}}\|f\|_{\sigma}\,\|g\|.
\end{split}
\end{equation}

For the second term in \eqref{supp:box-duality}, Lemma~\ref{lem:load-quadrature} gives
\begin{equation}\label{supp:box-second-term}
\left|
\left(f,(Q-I)L_h^{-\beta+\delta/2+1}w_h\right)
\right|
\lesssim
h^{1+\sigma}\|f\|_\sigma
\|L_h^{-\beta+\delta/2+1}w_h\|_{H^1(\cS)}.
\end{equation}
By norm equivalence and Proposition~\ref{prop:inverse-fractional}, applied with
\(r=3-2\beta+\delta\) and \(s=2\), we obtain
\[
\begin{aligned}
\left\|L_h^{-\beta+\delta/2+1}w_h\right\|_{H^1(\mathcal S)}
&\lesssim
\|w_h\|_{3-2\beta+\delta,h} \\
&\lesssim
h^{\min\{2\beta-\delta,1\}-1}\|w_h\|_{2,h}.
\end{aligned}
\]
Substituting this into \eqref{supp:box-second-term} and then using \eqref{supp:box-ritz}, we find
\begin{equation}\label{supp:box-second-term-final}
\begin{split}
\left|
\left(f,(Q-I)L_h^{-\beta+\frac{\delta}{2}+1}w_h\right)
\right|
&\lesssim
h^{1+\sigma}
h^{\min\{2\beta-\delta,1\}-1}
\|w_h\|_{2,h}
\|f\|_\sigma\lesssim
h^{\min\{2\beta+\sigma-\delta,1+\sigma\}}
\|f\|_\sigma \|g\|.
\end{split}
\end{equation}

Substituting \eqref{supp:box-first-term-final} and \eqref{supp:box-second-term-final} into \eqref{supp:box-duality}, we obtain
$
|(g,d_h)|
\lesssim
h^{\min\{2\beta+\sigma-\delta,1+\sigma\}}
\|f\|_\sigma \|g\|.
$
Since this holds for every $g\in L^2(\cS)$, duality gives
$
\|d_h\|
\lesssim
h^{\min\{2\beta+\sigma-\delta,1+\sigma\}}
\|f\|_\sigma.
$
Recalling that $\|u_{\beta,h}-\bar u_{\beta,h}\|_\delta\lesssim \|d_h\|$, the result follows.
\end{proof}

\section{Numerical experiments}\label{sec:numerics}

In this section, we compare our theoretical findings with the numerical results. In Sections~\ref{exp:6.1} and~\ref{exp:6.3} we take $L=-\Delta+I$ ($A=I$, $\kappa\equiv 1$), while in Section~\ref{exp:6.2} we take $L=-\Delta$ ($A=I$, $\kappa\equiv 0$). The one-dimensional tests use uniform $P_1$ meshes of $[0,1]$ with $n=2^{k}$ elements, $k=3,\ldots,9$, and homogeneous Dirichlet conditions; the two-dimensional test is a Neumann problem on uniform lattices with $(2^{k}+1)^{2}$ vertices, $k=3,\ldots,7$, compared with an overkill reference at level $k=8$. Throughout, the quantity labeled ``DOFs'' on the horizontal axes is $2^{k}+1$ (mesh vertices in one dimension, including boundary nodes; vertices along each side in two dimensions), so the free interior unknowns in the Dirichlet experiments are $2^{k}-1$. Fractional powers are approximated by the sinc quadrature of \cite{bonito2015numerical}, all errors are measured in $L^{2}(\cS)$ ($\delta=0$), and the codes are available in \cite{frac_box_exp}.

\subsection{Indicator load in one dimension}\label{exp:6.1}

We first consider $\cS=[0,1]$ with right-hand side $f=\mathbbm{1}_{[0,1/2]}$. Then $f\in \dot H^\gamma$ for every $\gamma<1/2$, whereas $f\notin \dot H^\gamma$ for $\gamma\ge 1/2$. Moreover, on a uniform one-dimensional mesh, one has $(f,\phi_z)=(f,Q\phi_z)$ for every nodal basis function $\phi_z$. Therefore, the FEM and box-method error predictions coincide, and both yield the rate $\min\{2\beta+1/2,2\}$. The exact solution $u_{\beta}$ is represented by the eigenbasis expansion
\[
u_{\beta}(x)=\sum_{n=1}^{\infty}
\frac{2\left(1-\cos(n\pi/2)\right)}{n\pi\left((n\pi)^{2}+1\right)^{\beta}}\sin(n\pi x).
\]
In the computations, the reference solution was approximated by truncating this series after a $\beta$-dependent number of modes (up to $8000$ terms for the smallest values of $\beta$), chosen large enough that the truncation error does not affect the experimental order of convergence (EOC). Figure~\ref{fig:indicator1d} shows that both discretizations attain the theoretical order of convergence (TOC) predicted by Corollary~\ref{cor:fully-implementable}; the coincidence of the two curves reflects that, for this piecewise constant load on a uniform one-dimensional mesh, the box average matches standard cellwise quadrature.

\begin{figure}[t]
\centering
\includegraphics[width=0.90\textwidth]{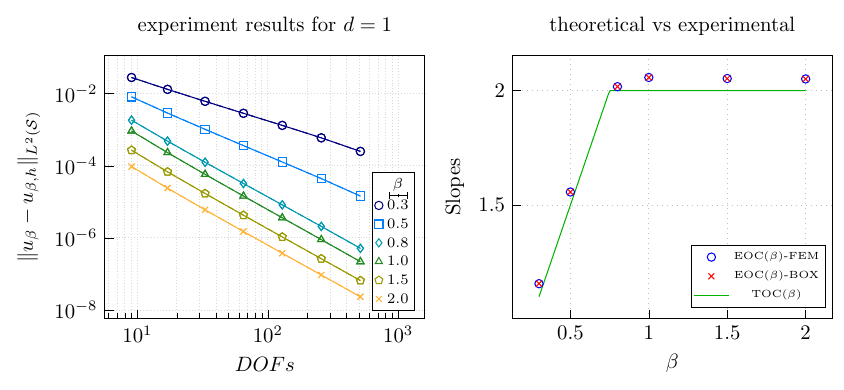}
\caption{One-dimensional indicator load. Left: $L^{2}(\cS)$ error versus DOFs. Right: experimental orders of convergence compared with $\min\{2\beta+1/2,2\}$.}
\label{fig:indicator1d}
\end{figure}

\subsection{A singular power load in one dimension}\label{exp:6.2}

Our second one-dimensional test uses $f(x)=x^{-0.499}$ on $\mathcal{S}=(0,1)$. In this case, $f\in \dot H^\gamma$ for every $\gamma<0.001$, whereas $f\notin \dot H^\gamma$ for $\gamma\ge 0.001$. The reference solution is computed by truncating the eigenfunction expansion of $u_\beta$:
\begin{equation*}
u_{\beta}(x)\approx 2\sum_{n=1}^{8000}
(n\pi)^{-2\beta}
\left[\int_{0}^{1} x^{-0.499}\sin(n\pi x)\,dx\right]\sin(n\pi x).
\end{equation*}
For $\beta=1$, classical box-method theory on uniform meshes predicts a rate not worse than $h^{1.5}$ for this type of singular load; see \cite{ewing2002}. Figure~\ref{fig:xalpha1d} confirms that the observed rates remain consistent with the abstract theory: the saturation at larger $\beta$ is limited by the spatial singularity of $f$, and the box method stays strictly below the FEM order when $\left(f,\phi_z\right)\neq\left(f,Q\phi_z\right)$.

\begin{figure}[t]
\centering
\includegraphics[width=0.90\textwidth]{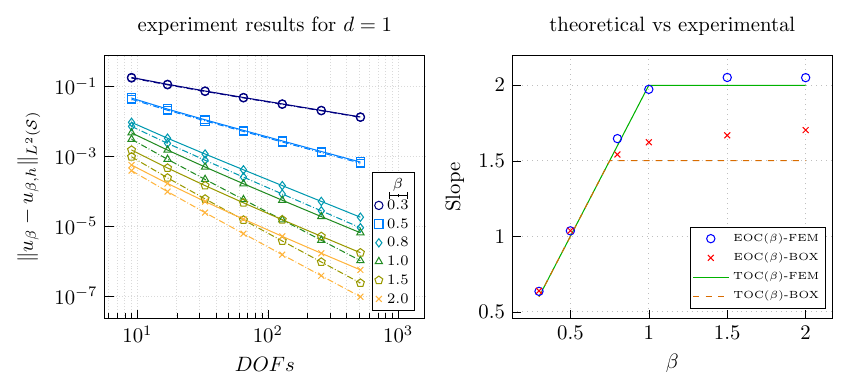}
\caption{One-dimensional singular load $f(x)=x^{-0.499}$. Left: $L^{2}(\cS)$ error versus DOFs. Right: experimental slopes compared with the theoretical prediction.}
\label{fig:xalpha1d}
\end{figure}

\subsection{Changing inner products in two dimensions}\label{exp:6.3}
On the unit square $\mathcal{S}=[0,1]^2$ we take
$$f(x,y)=-\sign\left(\left(x-\tfrac12\right)\left(y-\tfrac12\right)\right),$$
which belongs to $\dot H^{\gamma}$ for every $\gamma<1/2$. Here the interest is the dependence on the auxiliary inner product rather than the FEM/box comparison: we fix the box-method discretization and compare the exact $L^{2}(\cS)$ product with $\langle\cdot,\cdot\rangle_{0}$ and $\langle\cdot,\cdot\rangle_{1}$ from Proposition~\ref{prop:family-inner-products}. Figure~\ref{fig:indicator2d} shows that the observed slope is the same for all three choices, in agreement with Sections~\ref{sec:discretizations} and~\ref{sec:error}.

\begin{figure}[t]
\centering
\includegraphics[width=0.90\textwidth]{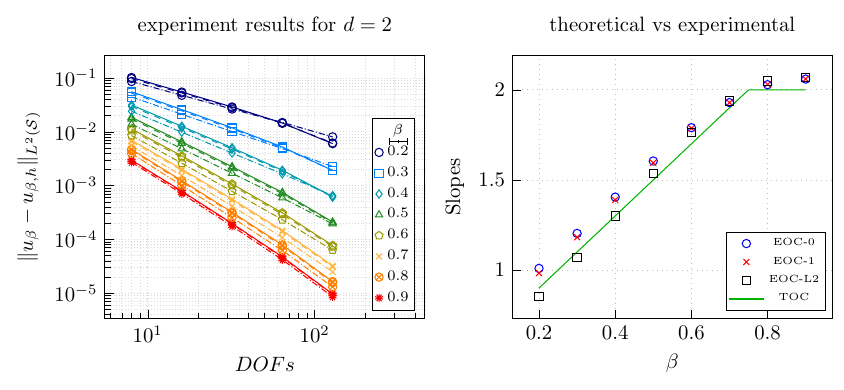}
\caption{Two-dimensional checkerboard load. Left: $L^{2}(\cS)$ error versus DOFs (here $2^{k}+1$, the number of vertices along each side). Right: experimental slopes compared with $\min\{2\beta+1/2,2\}$.}
\label{fig:indicator2d}
\end{figure}

\section*{Conflict of interest}
The authors declare that they have no conflict of interest.

\section*{Data availability}
The code and data used to produce the numerical experiments reported in Section~\ref{sec:numerics} are available in \cite{frac_box_exp}. No new empirical data were generated in the course of this study.

\bibliographystyle{siamplain}
\bibliography{sn-bibliography}

@article{du2020numerical,
  title={Numerical methods for nonlocal and fractional models},
  author={Du, Qiang and Gunzburger, Max and Lehoucq, Richard B. and Zhou, Kun},
  journal={Acta Numer.},
  volume={29},
  pages={1--124},
  year={2020}
}

@article{glusa2021error,
  title={Error estimates for the optimal control of a parabolic fractional {PDE}},
  author={Glusa, Christian and Ot{\'a}rola, Enrique},
  journal={SIAM J. Numer. Anal.},
  volume={59},
  number={2},
  pages={1140--1165},
  year={2021},
  publisher={SIAM}
}

@article{caffarelli2007extension,
  title={An extension problem related to the fractional {Laplacian}},
  author={Caffarelli, Luis and Silvestre, Luis},
  journal={Comm. Partial Differential Equations},
  volume={32},
  number={8},
  pages={1245--1260},
  year={2007},
  publisher={Taylor \& Francis}
}

@article{stinga2017regularity,
  title={Regularity theory and extension problem for fractional nonlocal parabolic equations and the master equation},
  author={Stinga, Pablo Ra{\'u}l and Torrea, Jos{\'e} L.},
  journal={SIAM J. Math. Anal.},
  volume={49},
  number={5},
  pages={3893--3924},
  year={2017}
}

@article{harizanov2020numerical,
  title={Numerical approximation of fractional powers of elliptic operators},
  author={Harizanov, Svetoslav and Lazarov, Raytcho and Marinov, Pencho and Pasciak, Joseph and Vassilevski, Panayot},
  journal={IMA J. Numer. Anal.},
  volume={40},
  number={3},
  pages={1746--1792},
  year={2020}
}

@article{lindgren2011explicit,
  title={An explicit link between {Gaussian} fields and {Gaussian} {Markov} random fields: the stochastic partial differential equation approach},
  author={Lindgren, Finn and Rue, H{\aa}vard and Lindstr{\"o}m, Johan},
  journal={J. R. Stat. Soc. Ser. B. Stat. Methodol.},
  volume={73},
  number={4},
  pages={423--498},
  year={2011},
  publisher={Oxford University Press},
  tmpDOI= "10.1111/j.1467-9868.2011.00777.x"
}

@article{hollbacher2021sharp,
  title={A sharp interface method using enriched finite elements for elliptic interface problems},
  author={H{\"o}llbacher, Susanne and Wittum, Gabriel},
  journal={Numer. Math.},
  volume={147},
  number={4},
  pages={759--781},
  year={2021},
  publisher={Springer}
}

@article{hackbusch1989first,
  title={On first and second order box schemes},
  author={Hackbusch, Wolfgang},
  journal={Computing},
  volume={41},
  number={4},
  pages={277--296},
  year={1989},
  publisher={Springer}
}

@book{brenner,
author = "Brenner, Susanne C. and Scott, L. Ridgway ",
title="The Mathematical Theory of Finite Element Methods",
year="2008",
publisher="Springer New York",
address="New York, NY",
tmpDOI="10.1007/978-0-387-75934-0"
}

@article{bolin2020rational,
  title={The rational {SPDE} approach for {Gaussian} random fields with general smoothness},
  author={Bolin, David and Kirchner, Kristin},
  journal={J. Comput. Graph. Statist.},
  volume={29},
  number={2},
  pages={274--285},
  year={2020},
  publisher={Taylor \& Francis}
}

@article{bolin2024covariance,
 author = {Bolin, David and Simas, Alexandre B and Xiong, Zhen},
 tmpDOI = {10.1080/10618600.2023.2231051},
 journal = {J. Comput. Graph. Statist.},
 number = {1},
 pages = {64--74},
 publisher = {Taylor \& Francis},
 title = {Covariance--based rational approximations of fractional {SPDEs} for computationally efficient {Bayesian} inference},
 volume = {33},
 year = {2024}
}

@article{chatzipantelidis2002,
  title={Finite volume methods for elliptic {PDE}'s: a new approach},
  author={Chatzipantelidis, Panagiotis},
  journal={ESAIM Math. Model. Numer. Anal.},
  volume={36},
  number={2},
  pages={307--324},
  year={2002},
  publisher={EDP Sciences},
  tmpDOI = "10.1051/m2an:2002014"
}

@article{jianguo1998finite,
  title={On the finite volume element method for general self-adjoint elliptic problems},
  author={Jianguo, Huang and Shitong, Xi},
  journal={SIAM J. Numer. Anal.},
  volume={35},
  number={5},
  pages={1762--1774},
  year={1998},
  publisher={SIAM},
  tmpDOI = "10.1137/S0036142994264699"
}

@article{chou2000error,
  title={Error estimates in $\mbox{L}^{2}$,$\mbox{H}^{1}$ and $\mbox{L}^{\infty}$ in covolume methods for elliptic and parabolic problems: a unified approach},
  author={Chou, So-Hsiang and Li, Qian},
  journal={Math. Comp.},
  volume={69},
  number={229},
  pages={103--120},
  year={2000},
  tmpDOI = "10.1090/S0025-5718-99-01192-8"
}

@article{cox2020regularity,
  title={Regularity and convergence analysis in {Sobolev} and {H{\"o}lder} spaces for generalized {Whittle}--{Mat{\'e}rn} fields},
  author={Cox, Sonja G and Kirchner, Kristin},
  journal={Numer. Math.},
  volume={146},
  number={4},
  pages={819--873},
  year={2020},
  publisher={Springer},
  tmpDOI="10.1007/s00211-020-01151-x"
}

@book{ciarlet2002finite,
  title={The finite element method for elliptic problems},
  author={Ciarlet, Philippe G},
  year={2002},
  publisher={SIAM}
}

@book{Thomée2006,
author="Thom{\'e}e, Vidar",
title="{Galerkin} Finite Element Methods for Parabolic Problems",
year="2006",
publisher="Springer Berlin Heidelberg",
address="Berlin, Heidelberg",
tmpDOI="10.1007/3-540-33122-0_15"
}

@article{bank1987some,
  title={Some error estimates for the box method},
  author={Bank, Randolph E and Rose, Donald J},
  journal={SIAM J. Numer. Anal.},
  volume={24},
  number={4},
  pages={777--787},
  year={1987},
  publisher={SIAM},
  tmpDOI = "10.1137/0724050"
}

@article{bonito2015numerical,
  title={Numerical approximation of fractional powers of elliptic operators},
  author={Bonito, Andrea and Pasciak, Joseph},
  journal={Math. Comp.},
  volume={84},
  number={295},
  pages={2083--2110},
  year={2015},
  tmpDOI = "10.1090/S0025-5718-2015-02937-8"
}

@book{Haase2006,
author="Haase, Markus",
title="The Functional Calculus for Sectorial Operators",
year="2006",
publisher="Birkh{\"a}user Basel",
address="Basel",
tmpDOI="10.1007/3-7643-7698-8"
}

@inbook{Fujita_1991,
  title={Evolution problems},
  author={Fujita, Hiroshi and Suzuki, Takashi},
  booktitle={Finite Element Methods (Part 1)},
  publisher={Elsevier},
  year={1991},
  pages={789--928},
  issn={1570-8659},
  tmpurl={http://dx.doi.org/10.1016/s1570-8659(05)80043-2},
  tmpDOI={10.1016/s1570-8659(05)80043-2}
}

@article{voet2023mathematical,
  title={A mathematical theory for mass lumping and its generalization with applications to isogeometric analysis},
  author={Voet, Yannis and Sande, Espen and Buffa, Annalisa},
  journal={Comput. Methods Appl. Mech. Engrg.},
  volume={410},
  pages={116033},
  year={2023},
  publisher={Elsevier},
  tmpDOI = "10.1016/j.cma.2023.116033"
}

@article{ewing2002,
  title={On the accuracy of the finite volume element method based on piecewise linear polynomials},
  author={Ewing, Richard E and Lin, Tao and Lin, Yanping},
  journal={SIAM J. Numer. Anal.},
  volume={39},
  number={6},
  pages={1865--1888},
  year={2002},
  publisher={Society for Industrial \& Applied Mathematics (SIAM)},
  issn={1095-7170},
  tmpurl={http://dx.doi.org/10.1137/s0036142900368873},
  tmpDOI={10.1137/s0036142900368873}
}

@book{strangfix2008,
author = {Strang, Gilbert and Fix, George},
title = {An Analysis of the Finite Element Methods, New Edition},
publisher = {Wellesley-Cambridge Press},
year = {2008},
tmpDOI = {10.1137/1.9780980232707},
address = {Philadelphia, PA},
edition   = {},
tmpurl = {https://epubs.siam.org/doi/abs/10.1137/1.9780980232707},
tmpeprint = {https://epubs.siam.org/doi/pdf/10.1137/1.9780980232707}
}

@incollection{FIX1972525,
title = {Effects of quadrature errors in finite element approximation of steady state, eigenvalue and parabolic problems},
editor = {A.K. Aziz},
booktitle = {The Mathematical Foundations of the Finite Element Method with Applications to Partial Differential Equations},
publisher = {Academic Press},
pages = {525-556},
year = {1972},
isbn = {978-0-12-068650-6},
tmpDOI = {https://doi.org/10.1016/B978-0-12-068650-6.50024-1},
tmpurl = {https://www.sciencedirect.com/science/article/pii/B9780120686506500241},
author = {George J. Fix}
}

@article{Liang2001,
  author    = {Liang, Shengde and Ma, Xiuling and Zhou, Aihui},
  title     = {Finite Volume Methods for Eigenvalue Problems},
  journal   = {BIT Numer. Math.},
  year      = {2001},
  volume    = {41},
  number    = {2},
  pages     = {345--363},
  tmpDOI       = {10.1023/A:1021946607960},
  issn      = {1572-9125},
  tmpurl       = {https://doi.org/10.1023/A:1021946607960}
}

@book{Jin_2023, title={Numerical Treatment and Analysis of Time-Fractional Evolution Equations}, ISBN={9783031210501}, ISSN={2196-968X}, tmpurl={http://dx.doi.org/10.1007/978-3-031-21050-1}, tmpDOI={10.1007/978-3-031-21050-1}, journal={Appl. Math. Sci.}, publisher={Springer International Publishing}, author={Jin, Bangti and Zhou, Zhi}, year={2023} }

@misc{frac_box_exp,
  author = {Almeida-Sousa, Kelvin J. R. and Bolin, David and Simas, Alexandre B.},
  title = {{fractional\_box\_experiments}},
  year = {2026},
  howpublished = {\url{https://github.com/KJhonson/fractional_box_experiments}},
  note = {Accessed: 2026-05-10}
}

\end{document}